\documentclass{article}[10pt]
\usepackage{theorem}
\usepackage{eurosym}
\usepackage{amssymb}
\usepackage{float}
\usepackage{amsmath}
\usepackage{amsfonts}
\usepackage[american]{babel}
\usepackage{verbatim}
\usepackage{geometry}

\newtheorem{theorem}{Theorem}

\newtheorem{corollary}[theorem]{Corollary}

{\theorembodyfont{\rmfamily}

}

\newtheorem{lemma}[theorem]{Lemma}

\newtheorem{proposition}[theorem]{Proposition}
{\theorembodyfont{\rmfamily}
\newtheorem{remark}[theorem]{Remark}
}

\newenvironment{proof}[1][Proof]{\noindent\textbf{#1} }{\ \rule{0.5em}{0.5em}}

\newcommand*\re{\mathbb{R}}

\newcommand*\Omegabar{\overline{\Omega}}
\newcommand*\delomega{\partial\Omega}

\newcommand*\intomega{\int_{\Omega}}
\newcommand*\notO{\backslash\{0\}}
\newcommand*\WBOmega{\mathcal{B}_1(\Omega)}

\newcommand*\WBDiskrad{W^{1,2}_{0,rad}(B_1)\cap \mathcal{B}_1(B_1)}
\newcommand*\ovu{\overline{u}}
\newcommand*\dxb{{|x|^{\beta}}}

\begin{document}
\large{

\title{Extremal Functions for the Singular Moser-Trudinger Inequality in 2 Dimensions}
\maketitle
\date{}

\centerline{\scshape Gyula Csat\'{o} and Prosenjit Roy }
\medskip
{\footnotesize
 \centerline{Tata Institute of Fundamental Research, Centre For Applicable Mathematics, 560065 Bangalore, India}
   \centerline{csato@math.tifrbng.res.in, prosenjit@math.tifrbng.res.in}

}

\smallskip

\begin{abstract}
The Moser-Trudinger embedding has been generalized by Adimuthi and Sandeep to the following weighted version: if  $\Omega\subset\re^2$ is bounded, $\alpha>0$ and $\beta\in [0,2)$ are such that 
$$
  \frac{\alpha}{4\pi}+\frac{\beta}{2}\leq1,
$$
then
$$
  \sup_{\substack {v\in W^{1,2}_0(\Omega) \\ \|\nabla v\|_{L^2}\leq 1}}
  \int_{\Omega}\frac{e^{\alpha v^2}-1}{|x|^{\beta}}\leq C.
$$
We prove that the supremum is attained, generalizing a well-known result by Flucher, who has proved the case $\beta=0.$
\end{abstract}

\let\thefootnote\relax\footnotetext{\textit{2010 Mathematics Subject Classification.} Primary , Secondary .}
\let\thefootnote\relax\footnotetext{\textit{Key words and phrases.} Moser Trudinger embedding, extremal function.}

\section{Introduction}\label{section:introduction}

Let $\Omega\subset\re^2$ be a bounded open smooth set.
The Moser-Trudinger imbedding, which is due to Trudinger \cite{Trudinger} and in its sharp form to Moser \cite{Moser}, states that the following supremum is finite
$$
  \sup_{\substack {v\in W^{1,2}_0(\Omega) \\ \|\nabla v\|_{L^2}\leq 1}}
  \int_{\Omega}\left(e^{4\pi v^2}-1\right)<\infty.
$$
First it has been shown by Carleson and Chang \cite{Carleson-Chang} that the supremum is actually attained, if $\Omega$ is a ball. In \cite{Struwe 1}, Struwe proved that the result remains true if $\Omega$ is close to a ball in measure. Then Flucher \cite{Flucher} generalized this result to arbitrary domains in $\re^2.$ See also Malchiodi-Martinazzi \cite{Malchiodi-Martinazzi} and the references therein for some recent developments on the subject.

The Moser-Trudinger embedding has been generalized by Adimurthi-Sandeep \cite{Adi-Sandeep} to a singular version, which reads as the following:
If $\alpha>0$ and $\beta\in[0,2)$ is such that
\begin{equation}
 \label{intro:eq:alpha and beta sum}
  \frac{\alpha}{4\pi}+\frac{\beta}{2}\leq 1,
\end{equation}
then the following supremum is finite
$$
  \sup_{\substack {v\in W^{1,2}_0(\Omega) \\ \|\nabla v\|_{L^2}\leq 1}}
  \int_{\Omega}\frac{e^{\alpha v^2}-1}{|x|^{\beta}}<\infty\,.
$$
We prove in this paper the following theorem, which states that the supremum is also attained for the singular Moser-Trudinger embedding.

\begin{theorem}\label{theorem:intro:Extremal for Singular Moser-Trudinger}
Let $\Omega$ be a bounded open connected smooth subset of $\re^2,$ $\alpha>0$ and $\beta\in[0,2)$ be such that \eqref{intro:eq:alpha and beta sum} is satisfied.
Then there exists $u\in W_0^{1,2}\left(\Omega\right)$ such that $\|\nabla u\|_{L^2}\leq 1$ and
$$
  \sup_{\substack {v\in W^{1,2}_0(\Omega) \\ \|\nabla v\|_{L^2}\leq 1}}
  \int_{\Omega}\frac{e^{\alpha v^2}-1}{|x|^{\beta}}=\int_{\Omega}\frac{e^{\alpha u^2}-1}{|x|^{\beta}}.
$$
\end{theorem}

The essential difficulty is that the functional 
$$
  u\in W^{1,2}_0(\Omega)\to F_{\Omega}(u)=\intomega\frac{e^{\alpha u^2}-1}{|x|^{\beta}}
$$
is not continuous with respect to weak convergence. To see this fact one can take the usual Moser-sequence, as in Flucher \cite{Flucher} page 472.
The proof of this theorem follows the ideas of Flucher and is based on a concentration compactness alterantive by  Lions \cite{Lions}. We give here an outline of the proof, explaining what is new compared to Flucher's result and what is not. The proof is divided into 6 sections (Section \ref{section:introduction} is introduction).
\smallskip

\textit{Section \ref{section:notation}.} We introduce some notation and definitions and recall some of their properties.
\smallskip 

\textit{Section \ref{section:prelim. results}.} The concentration compactness alternative (see Theorem \ref{theorem:concentration alternative for singular moser trudinger}) applies to the new functional $F_{\Omega}$
without change and states that if a sequence $u_i$ does not concentrate at any point of $\Omega,$ then up to a subsequence $\lim_{i\to\infty}F_{\Omega}(u_i)=F_{\Omega}(u).$ We prove in this section that the hypothesis for the concentration compactness alternative is satisfied for the new singular functional $F_{\Omega}.$ This is essentially the same as for the Moser-Trudinger functional. However, we show that for the functional $F_{\Omega}$ it is sufficient to consider the case when $0\in\Omegabar$ and a maximizing sequence concentrates at $0$ (see Proposition \ref{proposition:if u_i concentrates somewhere else than zero}). 
\smallskip

\textit{Section \ref{section ball}.} We show that the supremum is attained if $\Omega$ is the ball by using the result of Carleson-Chang \cite{Carleson-Chang} (see Theorem \ref{theorem:Carleson and Chang strict inequality}) and the transformation introduced by Adimurthi-Sandeep \cite{Adi-Sandeep}, which relates $F_{\Omega}$ to the classical Moser-Trudinger functional for radial functions (see Lemma \ref{lemma:properties of T_a}). We will in particular deduce the following strict inequality (see Theorem \ref{theorem:supremum of FOmega on Ball})
\begin{equation}\label{intro:eq:strict ineq in ball}
  F^{\delta}_{B_1}(0)<F_{B_1}^{\sup},
\end{equation}
where $F^{\delta}_{B_1}(0)$ denotes the concentration level at $0$ and the right hand side denotes the supremum of $F_{B_1}.$ 

\smallskip

\textit{Section \ref{section:ball to domain}.} In this section we establish the inequality
\begin{equation}
 \label{intro:eq:ball to domain ineq.}
  F_{\Omega}^{\sup}\geq I_{\Omega}(0)^{2-\beta}F_{B_1}^{\sup},
\end{equation}
where $I_{\Omega}(0)$ is the conformal incenter of $\Omega$ at $0$ (see Theorem \ref{theorem:ball to general domain:sup inequality}). It consists of constructing for any given radial function $v$ on the ball a corresponding function given on $\Omega,$ which satisfies the estimate $F_{\Omega}(u)\geq I_{\Omega}(0)^{2-\beta}F_{B_1}(v).$ In this step there is a crucial difference with Flucher's result where the inquality is deduced from the isoperimetric inequality. To carry out the same construction we needed a singularly weighted isoperimetric inequality, which is on its own a deep result with many consequences. It has been established in a separate paper in Csat\'o \cite{Csato}.

\smallskip 

\textit{Section \ref{section:domain to ball}.} In this section we prove a reverse inequality to \eqref{intro:eq:ball to domain ineq.} for concentrating sequences: given a concentrating sequence $\{u_i\}$ at $0$ which maximizes the concentration level $F^{\delta}_{\Omega}(0),$ one can construct a sequence $v_i$ such that 
\begin{equation}\label{intro:eq reverse ineq:domain to ball}
  F_{\Omega}^{\delta}(0)=\lim_{i\to\infty}F_{\Omega}^{\delta}(u_i)\leq I_{\Omega}^{2-\beta}(0)\liminf_{i\to\infty}F_{B_1}(v_i)\leq I^{2-\beta}_{\Omega}(0)F_{B_1}^{\delta}(0),
\end{equation}
see Proposition \ref{proposition:main domain to ball}.
This will imply the concentration formula
\begin{equation}
 \label{intro:concentration formula}
  F_{\Omega}^{\delta}(0)=I_{\Omega}^{2-\beta}(0)F_{B_1}^{\delta}(0),
\end{equation}
see Theorem \ref{theorem:concentration formula by domain to ball}. An essential difference is that if $\beta=0,$ then $F_{\Omega}^{\delta}(x)=I_{\Omega}^2(x)F_{B_1}^{\delta}(0)$ for all $x\in\Omega,$ see Flucher \cite{Flucher}, which implies, a priori, that a maximizinig sequence can concentrate only at the point $x$ where $I_{\Omega}(x)$ is maximal. Clearly $I_{\Omega}(x)$ is independent of the functional. This is in strong contrast to our case ($\beta>0$), where \eqref{intro:concentration formula} holds only at zero, since $F_{\Omega}^{\delta}(x)=0$ if $x\neq 0.$ This is due to the dependence in $x$ of the integrand of the functional $F_{\Omega}.$ In particular, the map $x\mapsto F^{\delta}_{\Omega}(x)$ is not continuous, unless $\beta=0.$ The proof of \eqref{intro:eq reverse ineq:domain to ball} is long and technical. We made a great effort to give clear and rigorous proofs of all steps and our presentation differs significantly from Flucher's paper (see also Remarks \ref{remark:Flucher point 4 mistake} 
and \ref{remark:Flucher thesis with si}).

\smallskip

\textit{Section \ref{section:proof main theorem}.} We prove here Theorem \ref{theorem:intro:Extremal for Singular Moser-Trudinger}. Combining the inequalities \eqref{intro:concentration formula}, \eqref{intro:eq:strict ineq in ball} and then \eqref{intro:eq:ball to domain ineq.}, one obtains that
$$
  F_{\Omega}^{\delta}(0)<F_{\Omega}^{\sup}.
$$
One deduces from this strict inequality 
that a maximizing sequence cannot concentrate at $0.$ It now follows easily from the results of Section \ref{section:prelim. results} that the maximum is attained.

\section{Notations and Definitions}\label{section:notation}

Throughout this paper $\Omega\subset\re^2$ will denote a bounded open set with smooth boundary $\delomega.$ Its $2$-dimensional area is written as $|\Omega|.$ The $1$-dimensional Hausdorff measure is denoted by $\sigma.$ Balls with radius $R$ and center at $x$ are written $B_R(x)\subset\re^2;$ if $x=0,$ we simply write $B_R$. The space $W^{1,2}(\Omega)$ denotes the usual Sobolev space of functions and $W^{1,2}_0(\Omega)$ those Sobolev functions with vanishing trace on the boundary. Throughout this paper $\alpha,\beta\in \re$ are two constants satisfying $\alpha>0,$ $\beta\in[0,2)$ and 
$$
  \frac{\alpha}{4\pi}+\frac{\beta}{2}\leq1.
$$

$-$ We define the the funtctional $F_{\Omega},J_{\Omega}:W_0^{1,2}(\Omega)\to \re$ by
\begin{align*}
 F_{\Omega}(u)=&\int_{\Omega}\frac{e^{\alpha u^2}-1}{|x|^{\beta}}\,dx, \smallskip \\
 J_{\Omega}(u)=&\int_{\Omega}\left(e^{4\pi u^2}-1\right)dx. \smallskip \\
\end{align*}

$-$ We say that a sequence $\{u_i\}\subset W_0^{1,2}(\Omega)$ concentrates at $x\in\Omegabar$ if 
$$
  \lim_{i\to\infty}\|\nabla u_i\|_{L^2}=1\quad\text{ and }\quad\forall\;\epsilon>0\quad
  \lim_{i\to\infty}\int_{\Omega\backslash B_{\epsilon}(x)}|\nabla u_i|^2=0.
$$
This definition implies the convergence $|\nabla u_i|^2dx\rightharpoonup \delta_x$ weakly in meausre, where $\delta_x$ is the Dirac measure at $x.$   We will use the following well known property of concentrating sequences: if $\{u_i\}$ concentrates, then $u_i\rightharpoonup 0$ in $W^{1,2}(\Omega).$ In particular
\begin{equation} 
 \label{eq:properties of concentrating sequences}
  u_i\to 0\quad\text{ in }L^2(\Omega),
\end{equation}
see for instance Flucher \cite{Flucher} Step 1 page 478.
\smallskip

$-$ We define the sets
$$
  W^{1,2}_{0,rad}(B_1)=\left\{u\in W^{1,2}_0(B_1)\,\Big|\, u\text{ is radial }\right\}
$$
and analogously $C_{c,rad}^{\infty}(B_1)$ is the set of radially symmetric smooth functions with compact support in $B_1$.
By abuse of notation we will usually write $u(x)=u(|x|)$ for $u\in W^{1,2}_{0,rad}(B_1).$
Recall that $C_{c,rad}^{\infty}(B_1)$ is dense in 
$W^{1,2}_{0,rad}(B_1)$ in the $W^{1,2}$ norm. If in addition $u$ is radially decreasing we write $u\in W^{1,2}_{0,rad\searrow}(B_1),$ respectively $u\in C_{c,rad\searrow}^{\infty}(B_1).$
\smallskip

$-$ Define 
$$
  \mathcal{B}_1(\Omega)=\left\{u\in W^{1,2}_0(\Omega)\,\big|\,\|\nabla u\|_{L^2}\leq 1\right\}.
$$

\smallskip 

$-$ Finally we define
$$
F_{\Omega}^{\text{sup}}=\sup_{u\in \mathcal{B}_1(\Omega)}F_{\Omega}(u).
$$
$J_{\Omega}^{\text{sup}}$ is defined in an ananlogous way, replacing $F$ by $J.$
If $x\in\Omegabar$ and the supremum is taken only over concentrating sequences, we write $F_{\Omega}^{\delta}(x),$ more precisely
$$
  F_{\Omega}^{\delta}(x)=\sup\left\{\limsup_{i\to\infty}F_{\Omega}(u_i)\,\Big|\quad \{u_i\}\subset \mathcal{B}_1(\Omega)\text{ concentrates at } x\right\}.
$$
We define in an analogous way $J_{\Omega}^{\delta}(x).$ If $\Omega=B_1$, then we define
$$
  F_{B_1,rad\searrow}^{\text{sup}}=\sup_{u\in W^{1,2}_{0,rad\searrow}(B_1)\cap \mathcal{B}_1(B_1)}F_{B_1}(u),
$$
$$
  F^{\delta}_{B_1,rad\searrow}(0)=\sup\left\{\limsup_{i\to\infty}F_{B_1}(u_i)\,\Big|\quad \{u_i\}\subset W_{0,rad\searrow}^{1,2}(B_1)\cap \mathcal{B}_1(B_1)\text{ concentrates at } 0\right\}.
$$
We define $J_{B_1,rad\searrow}^{\text{sup}}$ and $J^{\delta}_{B_1,rad\searrow}(0)$ in an analogous way.
\smallskip

$-$ If $\Omega\subset\re^2$ then $\Omega^{\ast}$ is its symmetric rearrangement, that is $\Omega^{\ast}=B_R(0),$ where $|\Omega|=\pi R^2.$ If $u\in W_0^{1,2}(\Omega),$ then $u^{\ast}\in W^{1,2}_{0,rad\searrow}(B_R(0))$ will denote the Schwarz symmetrization of $u.$ For basic propertis of the Schwarz symmetrization we refer to Kesavan \cite{Kesavan}, Chapters 1 and 2, which we will use throughout.
In particular we will use frequently and without further comment that if $u\in W_0^{1,2}(\Omega),$ then $u^{\ast}$ satisfies 
$$
  F_{\Omega}(u)\leq F_{B_R}(u^{\ast})\quad\text{ and }\quad \|\nabla u^{\ast}\|_{L^2(B_R)}\leq \|\nabla u\|_{L^2(\Omega)}.
$$
We will additionally  need, as in Flucher \cite{Flucher}, a slight modification of the Hardy-Littlewood, respectively P\'olya-Szeg\"o theorem, stated in the next proposition.  

\begin{proposition}
\label{prop:Hardy littlewoood modified}
(i) Lef $f\in L^p(\Omega)$ and $g\in L^q(\Omega),$ where $1/p+1/q=1.$ Then for any $a\in \re$
$$
  \int_{\{f\geq a\}}f\,g\leq \int_{\{f^{\ast}\geq a\}}f^{\ast}g^{\ast}.
$$
\smallskip

(ii) Let $u\in W^{1,2}_0(\Omega)$ such that $u\geq 0.$ Then for any $t\in (0,\infty)$
$$
  \int_{\{u^{\ast}\leq t\}}|\nabla u^{\ast}|^2\leq \int_{\{u\leq t\}}|\nabla u|^2
  \quad\text{ and }\quad \int_{\{u^{\ast}\geq t\}}|\nabla u^{\ast}|^2\leq \int_{\{u\geq t\}}|\nabla u|^2.
$$
\end{proposition}
\smallskip

$-$ We say that a sequence of sets $\{A_i\}\subset \re^2$ are approximately small disks at $x\in\re^2$ (of radius $\tau_i$) as $i\to\infty$ if there exists  sequences $\tau_i,\sigma_i>0$ such that $\lim_{i\to\infty}\tau_i=0,$
$$
  \lim_{i\to\infty}\frac{\sigma_i}{\tau_i}=0
$$
and 
$$
  B_{\tau_i-\sigma_i}(x)\subset A_i\subset B_{\tau_i+\sigma_i}(x)\quad\text{ for all $i$ big enough.}
$$
\smallskip

$-$ If $x\in\Omega,$ then $G_{\Omega,x}$ will denote the Green's function of $\Omega$ with singularity at $x.$ It can always be decomposed in the form
$$
  G_{\Omega,x}(y)=-\frac{1}{2 \pi}\log(|x-y|)-H_{\Omega,x}(y),\qquad y\in\Omega\backslash\{x\},
$$
where $H$ is the regular part. The conformal incenter $I_{\Omega}(x)$ of $\Omega$ at $x$ is defined by
$$
  I_{\Omega}(x)=e^{-2\pi H_{\Omega,x}(x)}.
$$
We refer to Flucher \cite{Flucher} concerning properties and examples regarding the conformal incenter, cf. \cite{Flucher} Lemma 10 and Proposition 12 (see also \cite{Csato} Lemma 12). We will need in particular the following results:

\begin{proposition}
\label{proposition:properties of Green's function} Let $x\in\Omega.$ Then
$G_{\Omega,x}$ and $I_{\Omega}(x)$ have the following properties:

(a) For every $t\in[0,\infty)$
$$
  \int_{\{G_{\Omega,x}<t\}}\left|\nabla G_{\Omega,x}(y)\right|^2 dy=t.
$$

(b) For every $t\in [0,\infty)$
$$
  \int_{\{G_{\Omega,x}=t\}}\left|\nabla G_{\Omega,x}(y)\right| d\sigma=1.
$$

(c)
$$
  \lim_{t\to\infty}\frac{\left|\left\{G_{\Omega,x}>t\right\}\right|}{e^{-4\pi t}} =\pi \left(I_{\Omega}(x)\right)^2.
$$

(d) If $B_R=\Omega^{\ast}$ is the symmetrized domain, then 
$$
  I_{\Omega}(x)\leq I_{B_R}(0)=R.
$$

(e) If $t_i\geq 0$ is a given sequence such that $t_i\to\infty$, then the sets $\{G_{\Omega,x}>t_i\}$ are approximately small disks at $x$ of radius $\tau_i=I_{\Omega}(x)e^{-2\pi t_i}.$

\end{proposition}

\section{Some Preliminary Results}\label{section:prelim. results}

We first note that it is sufficient to work with non-negative smooth maximizing sequences. More precisely we have the following lemma, which we will use in Section \ref{section:domain to ball} in a crucial way.

\begin{lemma}
\label{lemma new:sup over W12 same as over Cinfty}
Let $\{u_i\}\subset\mathcal{B}_1(\Omega)$ be a sequence such that the limit $\lim_{i\to\infty}F_{\Omega}(u_i)$ exists. Then there exists a sequence $\{w_i\}\subset \mathcal{B}_1(\Omega)\cap C_c^{\infty}(\Omega)$ such that 
$$
  \liminf_{i\to\infty}F_{\Omega}(w_i)\geq \lim_{i\to\infty}F_{\Omega}(u_i).
$$
Moreover, if $u_i$ concentrates at $x_0\in\Omegabar,$ then also $w_i$ concentrates at $x_0$. In particular maximizing sequences for $F_{\Omega}^{\sup}$ and $F_{\Omega}^{\delta}(x_0)$ can always be assumed to be smooth and non-negative.
\end{lemma}

\begin{proof}
For each $i\in\mathbb{N}$ there exists $\{v_i^k\}\subset C_c^{\infty}(\Omega)$ such that $v_i^k\to u_i$ almost everywhere in $\Omega,$
$$
  v_i^k\to u_i\text{ in }W^{1,2}(\Omega)\text{ for }k\to\infty\quad \text{ and }\quad \|\nabla v_i^k\|_{L^2}=\|\nabla u_i\|_{L^2}\text{ for all }k.
$$
Using Fatou's lemma there exists $k_1(i)\in\mathbb{N}$ such that
$$
  F_{\Omega}(u_i)\leq F_{\Omega}(v_i^j)+\frac{1}{2^i}\quad\forall\, j\geq k_1(i).
$$
Moreover, from the convergence in $W^{1,2}(\Omega)$ we also obtan the existence of $k_2(i)\in\mathbb{N},$ such that
$$
  \|\nabla v_i^j-\nabla u_i\|_{L^2(\Omega)}\leq \frac{1}{2^i}\quad\forall\, j\geq k_2(i).
$$
We now define $w_i=v_i^{k(i)},$ where $k(i)=\max\{k_1(i),k_2(i)\}.$ It can be easily verified that $w_i$ has all the desired properties. 
\end{proof}

\smallskip

\begin{lemma}[compactness in interior]\label{lemma:compactness in the interior}
Let $0<\eta<1$ and suppose $\{u_i\}\subset W_0^{1,2}(\Omega)$ is such that
$$
  \limsup_{i\to\infty}\|\nabla u_i\|_{L^2}\leq \eta\quad\text{and}\quad 
  u_i\rightharpoonup u\text{ in }W^{1,2}(\Omega)
$$
for some $u\in W^{1,2}(\Omega).$ Then for some subsequence
$$
  \frac{e^{\alpha u_i^2}}{|x|^{\beta}}\to \frac{e^{\alpha u^2}}{|x|^{\beta}}\quad\text{ in }L^1(\Omega)
$$
and in particular $\lim_{i\to\infty}F_{\Omega}(u_i)=F_{\Omega}(u).$
\end{lemma}

\begin{proof} The idea of the proof is to apply Vitali convergence theorem.
We can assume that, up to a subsequence, that $u_i\to u$ almost everywhere in $\Omega$ and that
$$
  \|\nabla u_i\|_{L^2}\leq\theta=\frac{1+\eta}{2}<1\quad\forall\,i\in\mathbb{N}.
$$
We can therefore define $v_i=u_i/\theta\in \WBOmega,$
which satisfies $\|\nabla v_i\|_{L^2}\leq 1$ for all $i.$ Moreover let us define $\overline{\alpha}=\alpha\theta^2<\alpha,$ such that
$$
  \frac{\overline{\alpha}}{4\pi}+\frac{\beta}{2}<1.
$$
Let $E\subset\Omega$ be an arbitrary measurable set. We use H\"older inequality with exponents $r$ and $s,$ where
$$
  r=\frac{4\pi}{\overline{\alpha}}>1\quad\text{ and }\quad \frac{1}{s}=1-\frac{1}{r}> \frac{\beta}{2},
$$
to obtain that
$$
  \int_E\frac{e^{\alpha u_i^2}}{|x|^{\beta}}=
  \int_E\frac{e^{\overline{\alpha}v_i^2}}{|x|^{\beta}} \leq
  \left(\int_{E}e^{4\pi v_i^2}\right)^{\frac{1}{r}} \left(\int_E\frac{1}{|x|^{\beta s}}\right)^{\frac{1}{s}}.
$$
Let $\epsilon>0$ be given.
In view of the Moser-Trudinger inequality and using that $1/|x|^{\beta s}\in L^1(\Omega),$ we obtain that for any $\epsilon>0$ there exists a $\delta>0$ such that
$$
  \int_E\frac{e^{\alpha u_i^2}}{|x|^{\beta}}\leq \epsilon\quad\forall\,|E|\leq \delta\text{ and }i\in\mathbb{N}.
$$
This shows that the sequence $e^{\alpha u_i^2}/|x|^{\beta}$ is equi-integrable and the Vitali convergence theorem yields convergence in $L^1(\Omega).$
This proves the lemma.
\end{proof}

\begin{theorem}[Concentration-Compactness Alternative]\label{theorem:concentration alternative for singular moser trudinger}
Let $\{u_i\}\subset \mathcal{B}_1(\Omega).$ Then there is a subsequence and $u\in W_0^{1,2}(\Omega)$ with $u_i\rightharpoonup u$ in $W^{1,2}(\Omega),$  such that either

(a) $\{u_i\}$ concentrates at a point $x\in\Omegabar,$
\newline or

(b) the following convergence holds true
$$
  \lim_{i\to\infty}F_{\Omega}(u_i)=F_{\Omega}(u).
$$
\end{theorem}

\begin{proof}
This is a direct appliction of Theorem 1 in Flucher \cite{Flucher}. Lemma \ref{lemma:compactness in the interior} shows precisely that the hypothesis of $F_{\Omega}$ being compact in the interior is satisfied.  
\end{proof}

\begin{proposition}
\label{proposition:if u_i concentrates somewhere else than zero}
Let $\beta>0,$ $\{u_i\}\subset  \mathcal{B}_1(\Omega)$ and suppose that $u_i$ concentrates at $x_0\in\Omegabar,$ where $x_0\neq 0.$ Then 
one has that, for some subsequence, $u_i\rightharpoonup 0$ in $W^{1,2}(\Omega)$  and
$$
  \lim_{i\to\infty}F_{\Omega}(u_i)=F_{\Omega}(0)=0.
$$
In particular $F_{\Omega}^{\delta}(x_0)=0.$
\end{proposition}

\begin{remark}
\label{remark:F delta zero is not zero}
Note that $F_{\Omega}^{\delta}(0)>0.$ This follows by rescaling the Moser sequence (see Flucher \cite{Flucher} page 472) on a small ball around the origin and extending it by $0$ in $\Omega.$
\end{remark}

\begin{proof}
Since $x_0\neq 0,$ there exists an $\epsilon>0$ such that 
 $B_{\epsilon}(x_0)\cap B_{2\epsilon}(0)=\emptyset.$
We extend $u_i$ by $0$ in $\re^2\backslash \Omegabar$ and split the integral in the following way
\begin{align*}
  F_{\Omega}(u_i)
  =
  \int_{\Omega\backslash B_{\epsilon}(0)}\frac{e^{\alpha u_i^2}-1}{|x|^{\beta}} 
  +
  \int_{B_{\epsilon}(0)}\frac{e^{\alpha u_i^2}-1}{|x|^{\beta}}=A_i+B_i\,,
\end{align*}
$A_i$ can obviously be estimated by
$$
  A_i\leq\frac{1}{|\epsilon|^{\beta}}\intomega \left(e^{\alpha u_i^2}-1\right).
$$
Since $\beta>0,$ it follows that $\alpha<4\pi,$ and it follows easily from \eqref{eq:properties of concentrating sequences} and from Vitali convegence theorem (similarly as in the proof of Lemma \ref{lemma:compactness in the interior}) that  $\lim_{i\to\infty}A_i=0.$
We now show that $\lim_{i\to\infty}B_i=0.$
Choose $\eta\in C^{\infty}_c(\Omega)$ such that $\eta\geq 0$ and
$$
  \eta=1\quad\text{ in }B_{\epsilon}(0),\qquad \eta=0\quad\text{ in }\Omega\backslash B_{2\epsilon}(0).
$$
Define $w_i=\eta u_i\in W_0^{1,2}(B_{2\epsilon}(0))$ and note that
\begin{align*}
  \int_{B_{2\epsilon}(0)}|\nabla(\eta u_i)|^2\leq& 2\int_{B_{2\epsilon}(0)}\eta^2 |\nabla u_i|^2+2\int_{B_{2\epsilon}(0)}|\nabla \eta|^2 |u_i|^2 \smallskip
  \\
  \leq& 2\int_{\Omega\backslash B_{\epsilon}(x_0)}|\nabla u_i|^2 +C_{\eta}\int_{\Omega}|u_i|^2.
\end{align*}
Since $\{u_i\}$ concentrates at $0,$ we get that for some $i_0\in\mathbb{N}$
$$
  \int_{B_{2\epsilon}(0)}|\nabla(\eta u_i)|^2\leq \frac{1}{2}\quad\forall\,i\geq i_0\,.
$$
We can therfore apply Lemma \ref{lemma:compactness in the interior} to the sequence $\{\eta u_i\}$ and the domain $B_{2\epsilon}(0)$ to get that
\begin{align*}
  0\leq\lim_{i\to\infty}B_i\leq &
  \lim_{i\to\infty}\int_{B_{2\epsilon}(0)}\frac{e^{\alpha(\eta u_i)^2}-1}{|x|^{\beta}} =0,
\end{align*}
where we have used again \eqref{eq:properties of concentrating sequences}. 
\end{proof}
\smallskip

We first prove Theorem \ref{theorem:intro:Extremal for Singular Moser-Trudinger} for some simple cases, which is the content of the next proposition.

\begin{proposition}
\label{proposition:sup attained if 0 notin Omegabar}
There exists $u\in \mathcal{B}_1(\Omega)$ such that $F_{\Omega}(u)=F_{\Omega}^{\sup}$ in the following cases:
$$
\text{(i)}\quad0\notin\Omegabar\quad\text{ or }\quad
 \text{(ii)}\quad\frac{\alpha}{4\pi}+\frac{\beta}{2}<1.
$$
\end{proposition}

\begin{proof} Let $\{u_i\}\subset\WBOmega$ be a maximizing sequence, that is 
\begin{equation}\label{eq:proof thm sing mst in simply conn. max sequence}
  F_{\Omega}^{\sup}=\lim_{i\to\infty}F_{\Omega}(u_i).
\end{equation}
We can assume that $u_i\rightharpoonup u\in\WBOmega$ in $W^{1,2}(\Omega).$
\smallskip 

\textit{Part (i).} We can assume, in view of Flucher's result \cite{Flucher}, that $\beta>0,$ which implies that $\alpha<4\pi.$ Since $0\notin \Omegabar,$ there exists a constant $C_1$ such that $1/|x|^{\beta}\leq C_1$. Therefore we can proceed similarly as in the proof of Lemma \ref{lemma:compactness in the interior}, or Proposition \ref{proposition:if u_i concentrates somewhere else than zero}, by using H\"older inequality and Vitali convergence theorem
and obtain that
$$
  \lim_{i\to \infty}F_{\Omega}(u_i)=F_{\Omega}(u).
$$

\textit{Part (ii).} Let $\gamma>1$ be such that 
$$
  \gamma^2\frac{\alpha}{4\pi}=1-\frac{\beta}{2}.
$$
We set $\overline{\alpha}=\gamma^2\alpha$ which satisfies $\overline{\alpha}/(4\pi)+\beta/2=1.$
We define $v_i=u_i/\gamma,$ which satisfies $v_i\rightharpoonup v:=u/\gamma$ in $W^{1,2}(\Omega)$ and
$$
  \|\nabla v_i\|_{L^2}\leq \frac{1}{\gamma}<1\quad\forall\,i.
$$
We therefore get from Lemma \ref{lemma:compactness in the interior} that
$$
  \lim_{i\to\infty}\int_{\Omega}\frac{e^{\alpha u_i^2}}{|x|^{\beta}}=
  \lim_{i\to\infty}\int_{\Omega}\frac{e^{\overline{\alpha} v_i^2}}{|x|^{\beta}}
  =\int_{\Omega}\frac{e^{\overline{\alpha} v^2}}{|x|^{\beta}}
  =\int_{\Omega}\frac{e^{\alpha u^2}}{|x|^{\beta}},
$$
from which the statement of Part (ii) follows.
\end{proof}

\section{The Case $\Omega=B_1$.}\label{section ball}

In this section we deal with the case where $\Omega$ is the unit ball. The following lemma is essentailly due to Adimurthi-Sandeep \cite{Adi-Sandeep}.

\begin{lemma}
\label{lemma:properties of T_a}
Let $0<a<\infty,$ and $u$ be radial function on $B_1$. Define
$$
  T_a(u)(x)=\sqrt{a}u\left(|x|^{\frac{1}{a}}\right).
$$
Then $T_a$ satisfies that 
$$
  T_a:W_{0,rad}^{1,2}(B_1)\to W_{0,rad}^{1,2}(B_1).
$$ 
$T_a$ is invertible with $(T_a)^{-1}=T_{1/a}$ and it satisfies
$$
  \|\nabla(T_a(u))\|_{L^2}=\|\nabla u\|_{L^2}\quad \forall\,u\in W_{0,rad}^{1,2}(B_1).
$$
Moreover if $a=1-\frac{\beta}{2},$ then
\begin{equation}\label{lemma:eq: transformation Adi-Sandeep}
  F_{B_1}(u)=\frac{1}{a}J_{B_1}(T_a(u))+\frac{|B_1|}{a}-\int_{B_1}\frac{1}{|x|^\beta}\quad\forall\,u\in W_{0,rad}^{1,2}(B_1).
\end{equation}
\end{lemma}

\begin{proof}
\textit{Step 1.} Let us first show that $T_a$ satisfies
$$
  T_a\left(C_{c,rad}^{\infty}(B_1)\right)\subset W_{0,rad}^{1,2}(B_1).
$$
and 
\begin{equation}
\label{eq:proof of lemma properties of T_a grad transforms}
  \int_{B_1} |\nabla(T_a(u))|^2=\int_{B_1}|\nabla u|^2.
\end{equation}
Let $u\in C_{c,rad}^{\infty}({B_1})$ and $v=T_a(u).$ Since $u$ is bounded, so is also $v.$ In particular we immediately obtain that $v\in L^2({B_1}).$ Note that in general $v\notin C^1({B_1})$ (take for example $a<1$). However we will show that $v$ has weak derivatives. First note that $\nabla v$ is well defined on ${B_1}\backslash\{0\},$ since $v\in C^1({B_1}\backslash\{0\}).$ Moreover we have, by the change of variable $r=s^a$ that 
\begin{equation}\label{eq:proof lemma Ta:transformation r and s in nabla}
  \int_{B_1}|\nabla v|^2=  \frac{1}{a}\int_0^1 \left(u'\left(r^{\frac{1}{a}}\right)\right)^2r^{\frac{2}{a}-2}\,2\pi rdr 
  =\int_0^1 (u'(s))^2\,2\pi sds 
\end{equation}
This proves \eqref{eq:proof of lemma properties of T_a grad transforms}. Since $v\in C^1({B_1}\backslash\{0\}),$ we have that for any $\epsilon>0,$
$$
  \int_{{B_1}\backslash B_{\epsilon}(0)}v\frac{\partial\varphi}{\partial x_i}
  =
  -\int_{{B_1}\backslash B_{\epsilon}(0)}\frac{\partial v}{\partial x_i}\varphi 
  +
  \int_{\partial B_{\epsilon}(0)}v\varphi\nu_i\quad\forall\,\varphi \in C^{\infty}_{c}({B_1}),
$$
where $\nu=(\nu_1,\nu_2)$ is the unit outward normal on $\partial B_1$.
In view of \eqref{eq:proof of lemma properties of T_a grad transforms} we have that $\|\nabla v\|_{L^2}<\infty.$ Therefore, recalling also $\|v\|_{L^{\infty}}\leq\infty,$ we obtain by letting $\epsilon\to 0,$ that
$$
  \int_{{B_1}}v\frac{\partial\varphi}{\partial x_i}
  =
  -\int_{{B_1}}\frac{\partial v}{\partial x_i}\varphi \quad\forall\,\varphi \in C^{\infty}_{c}({B_1}).
$$
This shows that $v\in W^{1,2}_0({B_1}).$ We can therefore apply Pioncar\'e inequality to obtain a constant $C$ which satisfies
\begin{equation*}
  \|T_a(u)\|_{W^{1,2}}\leq C\|\nabla u\|_{L^2}\quad\forall\,u\in C^{\infty}_{c,rad}({B_1}).
\end{equation*}
We can now extend by a density argument $T_a$ to an operator defined on whole $W^{1,2}_{0,rad}(B_1).$

\smallskip
\textit{Step 2.} It remains to show \eqref{lemma:eq: transformation Adi-Sandeep}. Let again $v$ be given by $v=T_a(u).$ Recall that the assumptions on $a$ and $\beta$ imply that $a=\frac{\alpha}{4\pi}.$ Thus, by using the substitution $r=s^{1/a},$ we get
\begin{align*}
  F_{B_1}(u)+\int_{B_1}\frac{1}{|x|^{\beta}}=&\int_0^1\frac{e^{\alpha u(r)^2}}{r^{\beta}} 2\pi r dr=\frac{1}{a}\int_0^1e^{4\pi \big(\sqrt{a}u\big(s^\frac{1}{a}\big)\big)^2}ds  
  \smallskip \\
  =&
  \frac{1}{a}J_{B_1}(v)+\frac{1}{a}\int_{B_1} 1=\frac{1}{a}J_{B_1}(v)+\frac{|{B_1}|}{a}.
\end{align*}
This concludes the proof of the last statement.
\end{proof}
\smallskip

The following corollary follows easily from Lemma \ref{lemma:properties of T_a}.

\begin{corollary}
\label{corollary:supremums for rad on D for F and G are same}
Let $a=1-\beta/2.$ Then the following identities hold true
$$
  \sup_{u\in \WBDiskrad}F_{B_1}(u)=\frac{1}{a}\sup_{u\in \WBDiskrad}J_{B_1}(u)+\frac{|B_1|}{a}-\int_{B_1}\frac{1}{|x|^\beta},
$$
and
$$
  F_{B_1}^{\text{sup}}=\frac{1}{a}J_{B_1}^{\text{sup}}+\frac{|B_1|}{a}-\int_{B_1}\frac{1}{|x|^\beta}.
$$
\end{corollary}

\begin{proof}
The first equality follows directly from Lemma \ref{lemma:properties of T_a}. By Schwarz symmetrization, the two equalities of the corollary are  equivalent.
\end{proof}
\smallskip

One of the crucial ingredients of the proof is the following result of Carleson and Chang \cite{Carleson-Chang}. Essential is the strict inequality in the following theorem. The second equality is an immediate consequence of the properties of Schwarz symmetrization.

\begin{theorem}[Carleson-Chang]
\label{theorem:Carleson and Chang strict inequality}
The following strict inequality holds true
$$
  J_{{B_1},rad\searrow}^{\delta}(0)<J_{B_1,rad\searrow}^{\sup}=J^{\sup}_{B_1}\,.
$$
\end{theorem}

\begin{remark}
The result in Carleson and Chang is acutally more precise, stating that
$$
  \pi\,e=\sup_{x\in \overline{{B_1}}}J_{{B_1},rad\searrow}^{\delta}(x)<J_{B_1,rad\searrow}^{\sup}\,,
$$
but for our purpose we only need an estimate for the concentration level at $0.$ 
\end{remark}                    

From Lemma \ref{lemma:properties of T_a} and Theorem \ref{theorem:Carleson and Chang strict inequality} we easily deduce the following proposition.

\begin{lemma}
\label{lemma:strict ineq. between concentration level and supremum for sing. MT on disk}
Let $\{u_i\}\subset\mathcal{B}_1(B_1)$ be a sequence which concentrates at $0.$
If  $\{u_i^{\ast}\}$ also concentrates at $0,$ then the following strict inequality holds true
$$
  \limsup_{i\to\infty}F_{{B_1}}(u_i)<F_{B_1}^{\sup}\,.
$$
\end{lemma}

\begin{proof}
Let  $a=1-\beta/2$
and
define $v_i=T_a(u_i^{\ast}).$ Let us first show that $\{v_i\}$ concentrates at $0.$ From Lemma \ref{lemma:properties of T_a} we know that 
$\lim_{i\to\infty}\|\nabla v_i\|_{L^2}=\lim_{i\to\infty}\|\nabla u_i^{\ast}\|_{L^2}=1.$
Moreover, the same calculation which shows this identity, namely \eqref{eq:proof lemma Ta:transformation r and s in nabla}, shows also that for any $\epsilon>0$
$$
  \int_{{B_1}\backslash B_{\epsilon}(0)}|\nabla v_i|^2
  =
  \int_{{B_1}\backslash B_{\epsilon^{\frac{1}{a}}}(0)}|\nabla u_i^{\ast}|^2.
$$
It now follows that $v_i$ concentrates at $0,$ since $u_i^{\ast}$ does.
From the properties of the symmetric rearrangement, namely $F_{B_1}(u_i)\leq F_{B_1}(u_i^{\ast}),$ and from Lemma \ref{lemma:properties of T_a} we obtain that
\begin{align*}
  \limsup_{i\to\infty}F_{B_1}(u_i)\leq& \limsup_{i\to\infty}F_{B_1}(u_i^{\ast})
  =\frac{1}{a}\limsup_{i\to\infty} J_{B_1}(v_i)+\frac{|B_1|}{a}-\int_{B_1}\frac{1}{|x|^\beta}
  \smallskip \\
  \leq&\frac{1}{a} J^{\delta}_{B_1,rad\searrow}(0)+\frac{|B_1|}{a}-\int_{B_1}\frac{1}{|x|^\beta}.
\end{align*}
We now apply Theorem \ref{theorem:Carleson and Chang strict inequality} and Corollary  \ref{corollary:supremums for rad on D for F and G are same} to obtain that
$$
  \limsup_{i\to\infty}F_{B_1}(u_i)<\frac{1}{a}J_{B_1}^{\sup}+\frac{|B_1|}{a}-\int_{B_1}\frac{1}{|x|^\beta}=F_{B_1}^{\sup},
$$
which proves the proposition
\end{proof}
\smallskip

A consequence of Lemma \ref{lemma:strict ineq. between concentration level and supremum for sing. MT on disk} is the following theorem, stating that the supremum of $F_{B_1}$ is attained.

\begin{theorem}
\label{theorem:supremum of FOmega on Ball} The following strict inequality holds
$$
  F_{B_1}^{\delta}(0)<F_{B_1}^{\sup}.
$$
In particular there exists $u\in\mathcal{B}_1(B_1)$ such that $F_{B_1}^{\sup}=F_{B_1}(u).$
\end{theorem}

\begin{proof}
Let $\{u_i\}\subset\mathcal{B}_1(B_1)$ be a concentrating sequence at $0,$ which maximizes $F^{\delta}_{B_1}(0).$ Using \eqref{eq:properties of concentrating sequences} and the properties of symmetrization we obtain that $u_i^{\ast}\to 0$ in $L^2(B_1).$ Therefore $u_i^{\ast}$ must concentrate, otherwise we get from Theorem \ref{theorem:concentration alternative for singular moser trudinger} and Remark \ref{remark:F delta zero is not zero} the contradiction (taking again some subsequence)
$$
  0<F^{\delta}_{B_1}(0)=\lim_{i\to\infty}F_{B_1}(u_i)\leq \lim_{i\to\infty}F_{B_1}(u_i^{\ast})=F_{B_1}(0)=0.
$$
Thus $u_i^{\ast}$ must concentrate at $0.$ We can apply lemma \ref{lemma:strict ineq. between concentration level and supremum for sing. MT on disk} to obtain that
$$
  F_{B_1}^{\delta}(0)<F_{B_1}^{\sup}.
$$
The second statement of the theorem follows from this strict inequality, Theorem \ref{theorem:concentration alternative for singular moser trudinger} and Proposition \ref{proposition:if u_i concentrates somewhere else than zero}.
\end{proof}

\section{Ball to Domain Construction}\label{section:ball to domain}

In view of Proposition \ref{proposition:sup attained if 0 notin Omegabar}, it remains to prove Theorem \ref{theorem:intro:Extremal for Singular Moser-Trudinger} for general domain with $0\in\Omegabar,$ when $\alpha/(4\pi)+\beta/2=1,$ and we can also take $\beta>0.$
Hence from now on we always assume that we are in this case. In addition, we assume in this section and Section \ref{section:domain to ball} that $0\in\Omega.$ The ball to domain construction is given by the following defnition: for $v\in W^{1,2}_{0,rad}(B_1)$  and $x\in\Omega,$ define $P_x(v)=u:\Omega\backslash\{x\}\to \re$ by
\begin{equation*}
  P_x(v)(y)=v\left(e^{-2\pi G_{\Omega,x}(y)}\right)=v\left(\left(G_{B_1,0}\right)^{-1}(G_{\Omega,x}(y)\right),
\end{equation*}
where, by abuse of notation, we have identified $v$ and $G_{B_1,0}$ with the corresponding radial function.
The main result of this section is the following theorem.

\begin{theorem}
\label{theorem:ball to general domain:sup inequality}
For any  $v\in W^{1,2}_{0,rad}(B_1)\cap \mathcal{B}_1(B_1)$ define $u=P_0(v).$ Then $u\in   \mathcal{B}_1(\Omega)$ and it satisfies
$$
  F_{\Omega}(u)\geq I_{\Omega}(0)^{2-\beta}F_{B_1}(v).
$$
In particular the following inequality holds true
$$
  F_{\Omega}^{\text{sup}}\geq I_{\Omega}(0)^{2-\beta}F^{\sup}_{B_1}.
$$
Moreover if $\{v_i\}\subset W_{0,rad}^{1,2}(B_1)$ concentrates at $0,$ then $u_i=P_0(v_i)$ concentrates at $0.$
\end{theorem}

The proof of this theorem is based on the following theorem proven in Csat\'o \cite{Csato} Theorem 12, which is a consequence of a weighted isoperimetric inequality.

\begin{theorem}
\label{theorem:upper bound for R by Greens function with weight}
Let $\Omega$ be a bounded open smooth connected set with $0\in\Omega$ and let also $x\in\Omega.$ then the following inequality holds true 
$$
  |\Omega|^{1-\frac{\beta}{2}}\leq \frac{1}{4\pi^{1+\frac{\beta}{2}}} \int_{\delomega} \frac{1}{|y|^\beta|\nabla G_{\Omega,x}(y)|}d\sigma(y).
$$
\end{theorem}

Before proving Theorem \ref{theorem:ball to general domain:sup inequality} we prove several intermediate results. 

\begin{lemma}
\label{lemma:upper bound for conformal incenter with range of r}
Define the sets $S_r$ for $r\in (0,1]$ by
$$
  S_r=\left\{y\in\Omegabar:\,G_{\Omega,0}(y)=-\frac{1}{2\pi}\log(r)\right\}.
$$
Then the following inequality holds true
$$
  \left|I_{\Omega}(0)\right|^{2-\beta}\leq \frac{r^{\beta-2}}{4\pi^2 }\int_{S_r} \frac{1}{|y|^\beta|\nabla G_{\Omega,0}(y)|}d\sigma(y)\quad\forall\,r\in (0,1].
$$
\end{lemma}

\begin{remark}
The special case $\beta=0$ is exactly Theorem 17 in Flucher.
\end{remark}

\begin{proof} The set $S_r$ is the boundary of $A_r$ given by
$$
  A_r=\left\{y\in\Omegabar:\,G_{\Omega,0}(y)>-\frac{1}{2\pi}\log(r)\right\}.
$$
Note that $0\in A_r$ for all $r\in (0,1]$ and 
$$
  G_{A_r,0}(y)=G_{\Omega,0}(y)+\frac{1}{2\pi}\log(r) =-\frac{1}{2\pi}\log(|y|)-H_{A_r,0}(y),
$$
where $H_{A_r,0}$ is the regulart part of the Green's function $G_{A_r,0}.$ We thus get
$$
  H_{A_r,0}(y)=H_{\Omega,0}(y)-\frac{1}{2\pi}\log(r).
$$
From the definition of the conformal incenter we get
\begin{equation}\label{eq:I Ar equal r I}
  I_{A_r}(x)=e^{-2\pi H_{A_r,x}(x)}=r\,I_{\Omega}(x).
\end{equation}
Note also that by the strong maximum principle the sets $A_r$ are connected for all $r\in (0,1].$ Applying Theorem \ref{theorem:upper bound for R by Greens function with weight} to the domain $\Omega=A_r$ and $\delomega=S_r$ we get
$$
  |A_r|^{1-\frac{\beta}{2}}\leq \frac{1}{4\pi^{1+\frac{\beta}{2}}} \int_{S_r} \frac{1}{|y|^\beta|\nabla G_{\Omega,0}(y)|}d\sigma(y).
$$
Now we use that $|A_r|=\pi b_r^2$ for some $b_r>0.$ It follows from Proposition \ref{proposition:properties of Green's function} (d) and \eqref{eq:I Ar equal r I} that
$$
  r I_{\Omega}(0)=I_{A_r}(0)\leq b_r=\sqrt{\frac{|A_r|}{\pi}}.
$$
Setting this into the previous inequality gives
$$
  |r I_{\Omega}(0)|^{2-\beta}\leq\left(\frac{|A_r|}{\pi}\right)^{1-\frac{\beta}{2}}\leq 
  \frac{1}{4\pi^2} \int_{S_r} \frac{1}{|y|^\beta|\nabla G_{\Omega,0}(y)|}d\sigma(y),
$$
from which the lemma follows.
\end{proof}

The following lemma holds true for any domains, whether containing the origin or not. So we state this general version, although we will use it with $x=0.$

\begin{lemma}
\label{lemma:ball to domain via G:preserves dirichlet norm} Let $x\in\Omega$ and let $v\in W^{1,2}_{0,rad}(B_1).$
Then $P_x(v)\in W^{1,2}_0(\Omega)$ and in particular 
\begin{equation}\label{lemma:eq:nablu u is nabla Pu}
  \|\nabla (P_x(v))\|_{L^2(\Omega)}=\|\nabla v\|_{L^2(B_1)}.
\end{equation}
Moreover if $\{v_i\}\subset W_{0,rad}^{1,2}(B_1)$ concentrates at $0,$ then $P_x(v_i)$ concentrates at $x.$
\end{lemma}

\begin{proof} \textit{Step 1.} We write $G=G_{\Omega,x}.$
Let $h$ be defined by $h(y)=e^{-2\pi G(y)}$ and hence $u(y)=v(h(y)).$ In particular
$$
  \nabla u(y)=v'(h(y))\,\nabla h(y).
$$
Note that, since $G\geq 0$ in $\Omega$ we get that if $y\in h^{-1}(t)\cap \Omega,$ then $t\in [0,1].$ Thus the coarea formula gives that
\begin{align*}
 \intomega |\nabla u|^2=&\intomega |v'(h(y))|^2 |\nabla h(y)|\,|\nabla h(y)| dy 
 \smallskip \\
 =&\int_0^1\left[\int_{h^{-1}(t)\cap \Omega}|v'(h(y))|^2|\nabla h(y)|d\sigma(y)\right]dt.
\end{align*}
Using that $|\nabla h(y)|=2\pi h(y)|\nabla G(y)|,$ gives
$$
  \intomega |\nabla u|^2=\int_0^1 2\pi t |v'(t)|^2\left[\int_{h^{-1}(t)\cap\Omega}|\nabla G(y)|d\sigma(y)\right]dt.
$$
Note that
$$
  h^{-1}(t)\cap \Omega=\left\{y\in\Omega\,\Big|\, G(y)=-\frac{1}{2\pi}\log(t)\right\}.
$$
Thus we obtain from Proposition \ref{proposition:properties of Green's function} (b) that
$$
  \int_{h^{-1}(t)\cap \Omega}|\nabla G(y)|d\sigma(y)=1\quad\forall\,t\in(0,1),
$$
which implies that
$$
  \int_{\Omega}|\nabla u|^2=\int_0^1|v'(t)|^2 2\pi t\,dt=\int_{B_1}|\nabla v|^2.
$$
This proves \eqref{lemma:eq:nablu u is nabla Pu}.
\smallskip

\textit{Step 2.}
Let us now assume that $\{v_i\}$ concentrates at $0$ and let $\epsilon>0$ be given. We know from Proposition \ref{proposition:properties of Green's function} (e), that for some $M>0$ big enough $\{G>M\}\subset B_{\epsilon}(x).$ Thus we obtain exactly as in Step 1 that
$$
  \int_{\Omega\backslash B_{\epsilon}(x)}|\nabla u_i|^2\leq \int_{\{G\leq M\}}|\nabla u_i|^2=\int_{e^{-2\pi M}}^1|v_i'(t)|^22\pi t\,dt.
$$
The right hand side goes to $0,$ since $v_i$ concentrates. This proves that $u_i$ concentrates too.
\end{proof}

\smallskip

We are now able to prove the main theorem.

\begin{proof}[Proof (Theorem \ref{theorem:ball to general domain:sup inequality}).]
We abbreviate again $G=G_{\Omega,0}.$
From Lemma \ref{lemma:ball to domain via G:preserves dirichlet norm} we know that $u\in\mathcal{B}_1(\Omega).$ Using the coarea formula we get
\begin{align*}
 F_{\Omega}(u)=&\intomega\frac{\big(e^{\alpha u^2}-1\big)}{|y|^{\beta}}\frac{|\nabla G(y)|}{|\nabla G(y)|}dy
 =\int_0^{\infty}\left[\int_{G^{-1}(t)\cap\Omega}\frac{(e^{\alpha u^2}-1)}{|y|^{\beta}|\nabla G(y)|}d\sigma(y)\right]dt
 \smallskip 
 \\
 =&\int_0^{\infty}\left(e^{\alpha v^2(e^{-2\pi t})}-1\right)
 \left[\int_{S_{r(t)}}\frac{1}{|y|^{\beta}|\nabla G(y)|}d\sigma(y)\right]dt,
\end{align*}
where $r(t)=e^{-2\pi t}$ and
$S_{r(t)}$ is defined as in Lemma
\ref{lemma:upper bound for conformal incenter with range of r}. That lemma therefore gives us that
\begin{align*}
  F_{\Omega}(u)\geq &I_{\Omega}(0)^{2-\beta}\int_0^{\infty}\frac{ e^{\alpha v^2(r(t))}-1}{r(t)^{\beta}} (2\pi r(t))^2 dt 
  \smallskip \\
  =&-I_{\Omega}(0)^{2-\beta}\int_0^{\infty}\frac{ e^{\alpha v^2(r(t))}-1}{r(t)^{\beta}} 2\pi r(t)\,r'(t) dt
  \smallskip 
  \\
  =&I_{\Omega}(0)^{2-\beta}\int_0^{1}\frac{ e^{\alpha v^2(r)}-1}{r^{\beta}} 2\pi r\,dr=I_{\Omega}(0)^{2-\beta}F_{B_1}(v).
\end{align*}
This proves the first claim of the theorem. 
The statement about the concentration follows directly from Lemma \ref{lemma:ball to domain via G:preserves dirichlet norm}.
\end{proof}

\section{Domain to Ball Construction}\label{section:domain to ball}

The aim of this section is to prove the following theorem. Recall that we assume $0\in\Omega.$

\begin{theorem}[Concentration Formula]
\label{theorem:concentration formula by domain to ball}
The following formula holds
$$
  F_{\Omega}^{\delta}(0)=I_{\Omega}(0)^{2-\beta}F_{B_1}^{\delta}(0).
$$
\end{theorem}

The proof of this result will be a consequence of the following proposition, which allows to construct a concentrating sequence in the ball from a given concentrating sequence in $\Omega.$

\begin{proposition}\label{proposition:main domain to ball}
Let $\{u_i\}\subset  \mathcal{B}_1(\Omega)\cap C^{\infty}(\Omega)$ be a sequence which concentrates at $0$ and is a maximizing sequence for $F_{\Omega}^{\delta}(0).$
Then  there exists a sequence $\{v_i\}\subset W_{0,rad}^{1,2}(B_1)\cap \mathcal{B}_1(B_1)$ concentrating at $0,$ such that
\begin{equation*}
  F_{\Omega}^{\delta}(0)=\lim_{i\to\infty}F_{\Omega}(u_i)\leq I_{\Omega}^{2-\beta}(0)\liminf_{i\to\infty}F_{B_1}(v_i).
\end{equation*}
\end{proposition}

\begin{proof}[Proof (Theorem \ref{theorem:concentration formula by domain to ball}).]
From Lemma \ref{lemma new:sup over W12 same as over Cinfty} and Proposition \ref{proposition:main domain to ball} we immediately obtain that 
$$
  F_{\Omega}^{\delta}(0)\leq I_{\Omega}^{2-\beta}(0)F_{B_1}^{\delta}(0).
$$
The reverse inequality follows from Theorem \ref{theorem:ball to general domain:sup inequality}.
\end{proof}
\smallskip

The proof of Proposition \ref{proposition:main domain to ball} is long and technical. We split it into many intermediate steps. To make the presentation less cumbersome, we assume in what follows that $0\in\Omega.$ However, we actually need this, and the fact that concentration occurs at $0,$ only in Step 6 in the proof of Lemma \ref{main lemma si instead of 1. domain to ball}. We start with too auxiliary lemmas.

\begin{lemma}
\label{lemma:limit over Bri same as Omega for subsequence}
Suppose $\{u_i\}\subset \mathcal{B}_1(\Omega)$ concentrates at $x_0\in\Omega$ and let $\{r_i\}\subset\re$ be  such that $r_i>0$ for all $i$ and $\lim_{i\to\infty}r_i=0.$ Then there exists a subsequence $u_{j_i}$ such that 
$$
  \lim_{i\to\infty}F_{\Omega}(u_i)=
  \lim_{i\to\infty}\int_{\Omega}\frac{e^{\alpha u_{i}^2}-1}{|x|^{\beta}}dx
  =\lim_{i\to\infty}\int_{B_{2r_i}(x_0)}\frac{e^{\alpha u_{j_i}^2}-1}{|x|^{\beta}}dx.
$$
Moreover any subsequence of $u_{j_i}$ will also satisfy the above equality.
\end{lemma}

\begin{proof}
Define for each $i$ the functions $\eta_i\in W^{1,\infty}(\re^2)$ by
$$
  \eta_i(x)=\left\{\begin{array}{rl}
                    1 & \text{ if }x\in\re^2\backslash B_{2r_i}(x_0)
                    \smallskip \\
                    \frac{|x-x_0|}{ r_i} -1 &\text{ if }x\in B_{2r_i}(x_0)\setminus B_{r_i}(x_0)
                    \smallskip \\
                    0&\text{ if }x\in B_{r_i}(x_0).
                   \end{array}\right.
$$
Note that $|\nabla\eta_i|^2=1/r_i^2$ in $B_{2r_i}(x_0)\setminus B_{r_i}(x_0).$ We obtain for all $i,j$ that
\begin{align*}
 \label{eq:nabla etai u i L2 norm}
 \intomega|\nabla(\eta_iu_j)|^2\leq 2\intomega\left(|\nabla \eta_i|^2|u_j|^2+|\eta_i|^2|\nabla u_j|^2\right)=A_i(u_j)+B_i(u_j),
\end{align*}
where
$$
  A_i(u_j)=\frac{2}{r_i^2}\int_{B_{2r_i}(x_0)\setminus B_{r_i}(x_0)}|u_j|^2, \qquad B_i(u_j)=2\int_{\Omega\backslash B_{r_i}(x_0)}|\eta_i|^2|\nabla u_j|^2.
$$
Since $|\eta_i|\leq 1,$ $u_j$ concentrates at $x_0$ and  $u_j\to 0$ in $L^2,$ we get that for any fixed $i$ the following convergences hold true
$$
  \lim_{j\to\infty}A_i(u_j)=0\quad\text{ and }\quad\lim_{j\to\infty}B_i(u_j)=0.
$$
We can therefore chose a subsequence $u_{j_i}$ such that 
$$
  A_i(u_{j_i})+B_i(u_{j_i})\leq \frac{1}{2^i}.
$$
We finally set $v_i=\eta_iu_{j_i}\in W_0^{1,2}(\Omega),$ which satisfies by construction $\lim_{i\to\infty}\intomega |\nabla v_i|^2=0.$
We obtain that
$v_i\to 0$ in $W_0^{1,2}(\Omega).$
We can now apply Lemma \ref{lemma:compactness in the interior} to obtain that 
$$
\frac{e^{\alpha v_i^2}-1}{|x|^{\beta}}\to 0\quad\text{ in }L^1(\Omega).
$$
In particular, recalling that $\eta_i=1$ on $\re^2\backslash B_{2r_i}(x_0),$ we get
$$
  0=\lim_{i\to\infty}\int_{\Omega\backslash B_{2r_i}(x_0)} \frac{e^{\alpha v_i^2}-1}{|x|^{\beta}}=
  \lim_{i\to\infty}\int_{\Omega\backslash B_{2r_i}(x_0)} \frac{e^{\alpha u_{j_i}^2}-1}{|x|^{\beta}}.
$$
From this the statement of the lemma follows immediately.
\end{proof}
\smallskip

We will frequently use the following elementary Lemma.

\begin{lemma}
\label{lemma:a.e. convergence and bounded by 1}
Suppose $\{u_i\}$ is a sequence of measurable functions such that $u_i\to 0$ almost everywhere in $\Omega.$ Let $\{s_i\}\subset\re$ be a bounded sequence.
Then
$$
  \lim_{i\to\infty}\int_{\{u_i\leq s_i\}}\frac{e^{\alpha u_i^2}-1}{|x|^{\beta}}=0.
$$
\end{lemma}

\begin{proof}
Define the function $f_i$ by
$$
  f_i(x)=\frac{e^{\alpha u_i^2}-1}{|x|^{\beta}}\chi_i(x),
$$
where $\chi_i$ is the characteristic function of the set $\{u_i\leq s_i\}.$ Note that, since $\beta<2,$ we have that 
$|f_i|\leq g,$ for some $g\in L^1(\Omega).$ Therefore by dominated convergence theorem we get that $\lim_{i\to\infty}\int_{\Omega}f_i=0.$
\end{proof}

\begin{lemma}
\label{lemma:the capacity argument}
Suppose $\{u_i\}\subset\mathcal{B}_1(\Omega)\cap C^{\infty}(\Omega)$ concentrates at $0\in\Omega$ and satisfies
\begin{equation}
 \label{eq:lemma:capacity argument maxim seq.}
  \lim_{i\to\infty}F_{\Omega}(u_i)=F_{\Omega}^{\delta}(0).
\end{equation}
Then for any $r>0$ there exists $j\in\mathbb{N}$ and $k_j\in [1,2]$ such that 
\begin{equation}
 \label{eq:lemma capacity arg. not empty Br}
  \{u_j\geq k_j\}\cap B_r(0)\neq \emptyset
\end{equation}
and
all connected components $A$ of $\{u_j\geq k_j\}$ will have the property:
\begin{equation}
 \label{eq:lemma:capacity arg. Br and B2r}
  \text{If }\;A\cap B_r(0)\neq\emptyset\quad\text{ then }\quad A\subset B_{2r}(0).
\end{equation}
Moreover $A$ has smooth boundary.
\end{lemma}

\begin{proof} It is sufficient to prove that \eqref{eq:lemma capacity arg. not empty Br} and \eqref{eq:lemma:capacity arg. Br and B2r} hold with $k_j=1$ for some $j\in\mathbb{N}.$ This implies that \eqref{eq:lemma capacity arg. not empty Br} and \eqref{eq:lemma:capacity arg. Br and B2r} also hold for any $k\geq 1,$ and hence, using Sard's theorem, one can choose $k_j\in [1,2]$ appropriately such that $A$ has smooth boundary in addition.

First note that for all $n\in\mathbb{N}$ there exists a $j\geq n$ such that
\eqref{eq:lemma capacity arg. not empty Br} must hold. If this is not the case, then Lemma \ref{lemma:limit over Bri same as Omega for subsequence} and Lemma \ref{lemma:a.e. convergence and bounded by 1} imply that
$$
  \lim_{i\to\infty}F_{\Omega}(u_i)\leq \lim_{i\to \infty}\int_{B_r} \frac{e^{\alpha u_i^2}-1}\dxb 
  \leq\lim_{i\to\infty}\int_{\{u_i\leq 1\}}\frac{e^{\alpha u_i^2}-1}\dxb=0,
$$
which is a contradiction to \eqref{eq:lemma:capacity argument maxim seq.} (Recall that $F_{\Omega}^{\delta}(0)>0,$ see Remark \ref{remark:F delta zero is not zero}).

Suppose now that \eqref{eq:lemma:capacity arg. Br and B2r} does not hold. We show that this leads to a contradiction. In that case there exists for all $j\in\mathbb{N}$ a connected component $D_j$ of $\{u_j\geq 1\}$ and $a,b\in\Omega$ such that
$$
  a\in D_j\cap B_r\quad\text{ and }\quad b\in D_j\cap \Omega\backslash B_{2r}.
$$
For what follows we fix $j$ and omit the explicit dependence on $j$ (Note that $a$ and $b$ depend on $j$). Without loss of generality we can assume, by rotating the domain, that  $b=(b_1,0)$ and $b_1\geq 2r.$
Since $D_j$ is connected, for all $x_1\in [r,2r]$ there exists a $x_2\in\re$ such that $x=(x_1,x_2)\in D_j.$ In particular $u_j(x)\geq 1.$
Since $\Omega$ is bounded, there exists an $M>0,$ which is independent of the rotation of the domain (and hence of $j$), such that
$$
  \Omega\subset[-M,M]\times [-M,M]=[-M,M]^2.
$$
Let us extend $u_j$ by zero in $[-M,M]^2\backslash\Omega.$ We obtain in this way (using H\"older inequality in the last inequality) that for any $x_1\in[r,2r]$
\begin{align*}
 1\leq& u_j(x)=u_j(x_1,x_2)-u_j(x_1,-M) =\int_{-M}^{x_2}\frac{\partial u_j}{\partial x_2}(x_1,s)ds 
 \smallskip \\
 &\leq \int_{-M}^M\left|\frac{\partial u_j}{\partial x_2}(x_1,s)\right|ds 
 \leq \sqrt{2M}\left(\int_{-M}^M|\nabla u_j(x_1,s)|^2 ds\right)^{\frac{1}{2}}.
\end{align*}
Taking the square of the previous inequality and integrating  $x_1$ from $r$ to $2r$ gives
$$
  r\leq 2M\int_{r}^{2r}\int_{-M}^{M}|\nabla u_j(x_1,s)|^2 ds dx_1\leq 2M\int_{\Omega\backslash B_r}|\nabla u_j|^2.
$$
But this cannot hold true for all $j,$ since $u_j$ concentrates at $0$.
\end{proof}
\smallskip

The next lemma is about the first modification of the the sequence $\{u_i\}$ given in Proposition \ref{proposition:main domain to ball}.

\begin{lemma}
\label{lemma:main properties step 1-3 in Flucher}
Let $\{u_i\}\subset \mathcal{B}_1(\Omega)\cap C^{\infty}(\Omega) $ be a sequence which concentrates at $0\in\Omega$ and satisfies
$$
  \lim_{i\to\infty}F_{\Omega}(u_i)=F_{\Omega}^{\delta}(0).
$$
Then there exists  a sequence $\{v_i\}\subset \mathcal{B}_1(\Omega)$ and sequences $r_i>0,$ with $r_i\to 0$  and  $\{k_i\}\in [1,2]$
 such that 
\begin{align*}
 \{v_i\geq k_i\}\subset B_{2r_i},\quad \Delta v_i=0\quad\text{in }\{v_i<k_i\}.
\end{align*}
Moreover $v_i$ has the properties: there exist a sequence $\{\lambda_i\}\subset\re,$ $\lambda_i>0$ such that  
\begin{align*}
 &\text{(i)}\quad \lim_{i\to\infty}\lambda_i=\infty 
 \smallskip \\
 &\text{(ii)}\quad \lim_{i\to\infty}v_i(y)=0\quad\text{ for all $y$ in }\Omega\backslash\{0\} 
 \smallskip \\
 &\text{(iii)}\quad \lambda_i v_i\to G_{\Omega,0}\quad\text{in }C^2_{\text{loc}}(\Omega\backslash \{0\})
 \smallskip \\
 &\text{(iv)} \quad\lim_{i\to\infty}F_{\Omega}(v_i)=F_{\Omega}^{\delta}(0).
\end{align*}
\end{lemma}

\begin{proof}
\textit{Step 1.} Take a sequence of positive real numbers $r_i$ such that $\lim_{i\to\infty}r_i=0$ and choose a subsequence of $u_i$, using Lemma \ref{lemma:limit over Bri same as Omega for subsequence}, such that
\begin{equation}
 \label{eq:proof:lemma st13 Bri}
  F_{\Omega}^{\delta}(0)=\lim_{i\to\infty}F_{\Omega}(u_i)=\lim_{i\to\infty} \int_{B_{r_i}}\frac{e^{\alpha u_i^2}-1}\dxb.
\end{equation}
Choosing again a subsequence we can assume by Lemma \ref{lemma:the capacity argument} that there exist $k_i\in [1,2]$ such that all connected components $A$ of $\{u_i\geq k_i\}$ which intersect $B_{r_i}$ are contained in $B_{2r_i}.$ We define $A_i$ as the union of all such $A.$ We also know from Lemma \ref{lemma:the capacity argument} that $A_i$ is not empty. Let $w_i$ be the solution of (this exists because $A_i$ has smooth boundary)
$$
  \left\{\begin{array}{c}
          \Delta w_i=0\quad\text{in }\Omega\backslash \overline{A}_i 
          \smallskip \\
          w_i=0\quad\text{on }\delomega,\quad w_i=k_i\quad\text{on }\partial A_i.
         \end{array}\right.
$$
We now define $\overline{u}_i\in W^{1,2}_0(\Omega)$ as
$$
  \ovu_i=\left\{\begin{array}{rl}
               u_i&\text{ in }A_i \smallskip \\
               w_i&\text{ in }\Omega\backslash A_i\,.
              \end{array}\right.
$$
Since harmonic functions minimize the Dirichlet integral we have $\|\nabla\ovu_i\|_{L^2}\leq \|\nabla u_i\|_{L^2}.$ 
Thus we have constructed a sequence which has the properties:
\begin{align*}
 \{\ovu_i\geq k_i\}\subset B_{2r_i},\quad \Delta \ovu_i=0\quad\text{in }\{\ovu_i<k_i\}\quad\text{and}\quad
 \|\nabla \ovu_i\|_{L^2}\leq 1.
\end{align*}

\textit{Step 2.}  We will show in this Step that for all $y\in\Omega\backslash\{0\}$ we have $\ovu_i(y)>0$ for all $i$ large enough and $\lim_{i\to\infty}\ovu_i(y)=0.$
The fact that $\ovu_i(y)>0$ follows from the maximum principle.
Since $\Omega$ is bounded there exists $M>0$ such that $\Omegabar\subset B_M.$ Define $W_i=B_M\backslash\overline{B_{2r_i}}$ and let $\psi_i$ be the solution of 
$$
  \left\{\begin{array}{c}
          \Delta \psi_i=0\quad\text{ in }W_i 
          \smallskip \\
          \psi_i=2\;\text{ on }\partial B_{2r_i}\quad\text{ and }
          \quad\psi_i=0\;\text{ on }\partial B_M.
         \end{array}\right.
$$
The function $\psi_i$ can be given explicitly:
$$
  \psi_i=\frac{2}{\log\big(\frac{2r_i}{M}\big)}\log\left(\frac{|x|}{M}\right).
$$
Recall that $k_i\in[1,2]$ and note that
\begin{align*}
 &\psi_i>0\,\text{ and }\,\ovu_i=0\quad\text{ on }\delomega,
 \smallskip \\
 &\psi_i=2\,\text{ and }\,\ovu_i<k_i\leq 2\quad\text{ on }\partial B_{2r_i},
\end{align*}
and thus $\psi_i-\ovu_i>0$ on $\partial W_i$. Since $\ovu_i$ is also harmonic in $W_i$ the maximum principle implies that $\ovu_i\leq \psi_i$ in $W_i$.
For $i$ big enough $y\in W_i$ and the claim of Step 2 follows from the fact that $\lim_{i\to\infty}\psi_i(y)=0.$ 
\smallskip 

\textit{Step 3.} Choose $y\in \Omega\backslash\{0\}$ and define $\lambda_i$ by
\begin{equation}\label{eq:lemma proof:def of lambdai by uiy}
  \lambda_i=\frac{G_{\Omega,0}(y)}{\ovu_i(y)}\quad\Leftrightarrow\quad \lambda_i\ovu_i(y)=G_{\Omega,0}(y)
\end{equation}
In view of Step 2 this is well defined, $\lambda_i>0$ and
$$
  \lim_{i\to\infty}\lambda_i=\infty.
$$
Let $y\in K_1\subset \Omega\backslash\{0\}$ be a compact set. Choose another compact set $K_2,$ such that $K_1\subset\subset K_2\subset \Omega\backslash\{0\}.$
Applying Harnack inequality on $K_2$ we get that there exist $c_1,c_2>0,$ such that
$$
  c_1|G_{\Omega,0}(y)|\leq |\lambda_i \ovu_i(x)|\leq c_2|G_{\Omega,0}(y)|\quad\forall\,x\in K_2\,\text{ and }\,\forall\,i\text{ large enough}.
$$
Hence the sequence $\lambda_iu_i$ is uniformly bounded in the $C^0(K_2)$ norm. Choose $0<\alpha<1.$ It follows from Schauder estimates (see Gilbarg-Trudinger, Corollary 6.3 page 93 and the remark thereafter) that $\lambda_i u_i$ is also uniformly bounded in the $C^{2,\alpha}(K_1)$ norm. Using the compact embedding $C^{2,\alpha}(K_1)\hookrightarrow C^2(K_1)$ we obtain that there exists $g\in C^2(K_1)$ and a subsequence $\ovu_i$ with
$$
  \lambda_i\ovu_i\to g\quad\text{ in }C^2(K_1).
$$
We finally define $v_i=\ovu_i$ as this subsequence. It follows from \eqref{eq:lemma proof:def of lambdai by uiy} and Bocher's theorem (see for instance \cite{Axler-Bourdon-Ramey} Theorem 3.9 page 50) that $g=G_{\Omega,0}.$ 
\smallskip

\textit{Step 4.} It remains to prove (iv). Recall that $\ovu_i\leq k_i$ in $\Omega\backslash A_i$. We therefore obtain, using Lemma \ref{lemma:a.e. convergence and bounded by 1} twice and the definition of $A_i$ that
\begin{align*}
 &\lim_{i\to\infty}\intomega\frac{e^{\alpha \ovu_i^2}-1}\dxb
 =\lim_{i\to\infty}\int_{A_i}\frac{e^{\alpha u_i^2}-1}\dxb
 \geq \lim_{i\to\infty}\int_{A_i\cap B_{r_i}}\frac{e^{\alpha u_i^2}-1}\dxb
  \smallskip \\
 =&\lim_{i\to\infty}\int_{ \{u_i\geq k_i\}\cap B_{r_i}}\frac{e^{\alpha u_i^2}-1}\dxb=\lim_{i\to\infty}\int_{B_{r_i}}\frac{e^{\alpha u_i^2}-1}\dxb=F_{\Omega}^{\delta}(0),
\end{align*}
where we have used \eqref{eq:proof:lemma st13 Bri} in the last equality.
\end{proof}
\smallskip

The next lemma is about the second modification of the sequence $\{u_i\}$ given in Proposition \ref{proposition:main domain to ball}, following the first modification given by Lemma \ref{lemma:main properties step 1-3 in Flucher}.

\begin{lemma}
\label{lemma:the one we sent to Struwe}
Let $\{u_i\}\subset W_0^{1,2}(\Omega)$ be a sequence and $\lambda_i$ a sequence in $\re$ such that $\lambda_i\to\infty,$  
$$
  \lambda_i u_i\to G_{\Omega,0}\quad\text{ in }C^0_{loc}({\Omega}\backslash\{0\})\quad\text{ and }\quad \Delta u_i=0\text{ in }\{u_i<1\}.
$$
Then there exists a subsequence $\lambda_{i_l}$ and a sequence $\{v_l\}\subset W^{1,2}_0(\Omega)$ such that the following properties hold true:

(a) $\lambda_{i_l}\geq l$ 

(b) The sets $\left\{v_l\geq  l/\lambda_{i_l}\right\}$
are approximately small disks at $0$ as $l\to\infty.$

(c) $v_l(x)\to 0$ as $l\to\infty$ for every $x$ in $\Omega\backslash\{0\}.$

(d) For every $l$
$$
  \intomega|\nabla v_l|^2\leq\intomega|\nabla u_{i_l}|^2.
$$

(e) The inequality $v_l\geq u_{i_l}$ holds in $\Omega.$ In particular  
$F_{\Omega}( v_l)\geq F_{\Omega} ( u_{i_l}).$
\end{lemma}

\begin{remark}\label{remark:Flucher point 4 mistake} (i) Flucher in his paper \cite{Flucher} (see Point 4 page 492) claims that the hypothesis $\lambda_iu_i\to G_{\Omega,x_0}$ in $C^1_{loc}(\Omega\backslash\{0\})$ implies that for some subsequence the sets $\{u_i\geq 1\}$ form approximately small disks.
Lemma \ref{lemma:the one we sent to Struwe} actually shows that to obtain this property it is not sufficient to chose a subsequence, but the sequence has to be modified again.
This is necessary, as shown by the following example:
let $\Omega=B_1(0)$, $\lambda_i=i$ for all $i$ and define
$$
  u_i(x)=\left\{\begin{array}{rl}
                 -\frac{1}{2\pi i}\log(|x|) &\text{ if }|x|\geq e^{-i\pi} 
                 \smallskip \\
                 \frac{1}{2} &\text{ if }|x|\leq e^{-i\pi}.
                \end{array}\right.
$$
Obviously $\lambda_iu_i\to G_{B_1,0}$ in $C^{\infty}_{loc}(\Omega\backslash\{0\}).$
But the sets $\{u_i\geq 1\}$ are empty for all $i.$ One can easily construct an example where even $u_i(0)=0$ for all $i$ and not even the sets $\{u_i\geq s_i\}$ will have the desired property, for any sequence $s_i>0.$
\end{remark}

\begin{proof} Again we abbreviate $G=G_{\Omega,0}.$
\smallskip 

\textit{Step 1.} Let $l\in \mathbb{N}.$ We know from Proposition \ref{proposition:properties of Green's function} (e) that the sets $\{G\geq l\}$ form approximately small disks, that is $B_{\rho_l-\delta_l}\subset\{G\geq l\}\subset B_{\rho_l+\delta_l},$
for two sequences $\rho_l$ and $\delta_l$ tending to zero and satisfying
$
  \lim_{l\to\infty}({\delta_l}/{\rho_l})=0.
$
We know from maximum principle for harmonic functions, respectively Hopf lemma that the following strict inequalities hold 
$$
  G<l\quad\text{ in }\Omegabar\backslash B_{\rho_l+2\delta_l}
  \quad\text{ and }\quad
  G>l\quad\text{ on }\partial B_{\rho_l-2\delta l}.
$$
Let $\epsilon_1>0$ be such that
$G\leq l-\epsilon_1$ on $\Omegabar\backslash B_{\rho_l+2\delta_l}.$
Using now the locally uniform convergence of the hypothesis, we know that there exists $j_l\in\mathbb{N}$ such that
$$
  \|\lambda_i u_i-G\|_{C^0(\Omegabar_{\eta}\backslash B_{\rho_l+2\delta_l})}\leq \epsilon_1\quad\forall\,i\geq j_l,
$$
where $\eta>0$ and $\Omega_{\eta}=\{x\in\Omega|\,\operatorname{dist}(x,\delomega)>\eta\}$ is choosen such that $\overline{B}_{\rho_l+2\delta_l}\in\Omega_{\eta}$ for all $l.$
In particular this implies that
\begin{equation}\label{new lemma proof:lambda i u_i smaller l}
  \lambda_i u_i\leq l\quad\text{ in }\Omegabar_{\eta}\backslash B_{\rho_l+2\delta_l}\quad\forall\,i\geq j.
\end{equation}
Let $\epsilon_2>0$ be such that $G\geq l+\epsilon_2$ on $\partial B_{\rho_l-2\delta_l}.$
We use again locally uniform convergence and choose $i_l\geq j_l$ such that 
$$
  \|\lambda_{i_l}u_{i_l}-G\|_{C^0(\partial B_{\rho_l-2\delta_l})}\leq \epsilon_2.
$$
Moreover, since $\lambda_i\to\infty$ we can assume, by choosing $i_l$ if necessary even larger, that
\begin{equation}\label{lambda il is bigger l}
  \lambda_{i_l}\geq l.
\end{equation}
In particular we obtain that
\begin{equation}
 \label{lambda il u il is bigger than l on partial B}
  \lambda_{i_l} u_{i_l}\geq l\quad\text{ on }\partial B_{\rho_l-2\delta_l}.
\end{equation}
Finally we define the set $A_l$ by
\begin{equation}
 \label{eq:definition of sets A_l}
  A_l=\left\{x\in B_{\rho_l-2\delta_l}:\;\lambda_{i_l}u_{i_l}<l\right\}.
\end{equation}
At last we define
\begin{equation}
 \label{definition of w il}
  v_l=\left\{\begin{array}{rl}
  u_{i_l}&\text{ in }\Omegabar\backslash A_l \smallskip \\
  \frac{l}{\lambda_{i_l}}&\text{ in }A_l\,.
  \end{array}\right.
\end{equation}
Note that $v_l\in W^{1,2}_0(\Omega)$ in view of  \eqref{lambda il u il is bigger than l on partial B} and \eqref{eq:definition of sets A_l}.
\smallskip

\textit{Step 2.} Let us now verify that the new sequence $v_l$ verifies (a)--(e). The first statement (a) is satisfied obviously by \eqref{lambda il is bigger l}. Using the hypothesis $\Delta u_i=0$ in $\{u_i<1\}$ and again \eqref{lambda il is bigger l} we get that $\Delta u_{i_l}=0$ in $\{u_{i_l}<l/\lambda_{i_l}\}.$ Thus from the maximum principle we also have
$$
  \lambda_{i_l}u_{i_l}\leq l\quad\text{ in }\Omegabar\backslash\Omega_{\eta}.
$$
Together with \eqref{new lemma proof:lambda i u_i smaller l} we get
$$
  \left\{\lambda_{i_l}u_{i_l}\geq l\right\}\subset B_{\rho_l+2\delta_l},
$$
which implies, by the definition of $v_l$, that 
$\left\{v_l\geq {l}/{\lambda_{i_l}}\right\}\subset B_{\rho_l+2\delta_l}.$
Moreover from the definition of $A_l$ and the definition of $v_l$ we get that $B_{\rho_l-2\delta_l}\subset\left\{v_l\geq {l}/{\lambda_{i_l}}\right\}.$
This shows (b) indeed:
$$
  B_{\rho_l-2\delta_l}\subset\left\{v_l\geq \frac{l}{\lambda_{i_l}}\right\}\subset B_{\rho_l+2\delta_l}.
$$
Let us now show (c). Let $x\in \Omega\notO.$ Then for all $l$ big enough we get that
$x\in\Omega\backslash B_{\rho_l+2\delta_l}.$ So for those $l$ we have $v_l=u_{i_l}.$ Since we have that $\lambda_{i_l}\to\infty$ and that
$$
  \lambda_{i_l}u_{i_l}(x)\to G(x),
$$
we must have that $u_{i_l}(x)\to 0,$ which proves (c). The statement (d) follows immediately from the definition \eqref{eq:definition of sets A_l} of $v_l$. (e) follows also directly from \eqref{eq:definition of sets A_l} and \eqref{definition of w il}.
\end{proof}

After having modified the sequence $\{u_i\}$ given in Proposition \ref{proposition:main domain to ball} in the two previous lemmas, we finally construct the appropriate corresponding sequence $\{v_i\}\subset W_{0,rad}^{1,2}(B_1).$  This is contained in the following lemma.

\begin{lemma}\label{main lemma si instead of 1. domain to ball}
Let $\{u_i\}\subset W_0^{1,2}(\Omega)$ and $\{s_i\}\subset\re$ be sequences with the following properties:
$$
  s_i\leq 1\quad\forall\,i\in\mathbb{N},
$$
the sets $\{u_i\geq s_i\}$ are approximately small disks at $0$ as $i\to\infty$ and
moreover suppose that pointwise $u_i(x)\to 0$ for all $x\in\Omega\backslash\{0\}.$
Then there exists a sequence $\{v_i\}\subset W^{1,2}_{0,rad}(B_1)$ such that for all $i$
$$
  \|\nabla v_i\|_{L^2(B_1)}\leq \|\nabla u_i\|_{L^2(\Omega)}
$$
and, assuming that the left hand side limit exists, 
$$
  \lim_{i\to\infty}F_{\Omega}(u_i)\leq 
  I_{\Omega}(0)^{2-\beta}\lim\inf_{i\to\infty}F_{B_1}(v_i).
$$
Moreover $v_i(x)\to 0$ for all $x\in B_1\backslash \{0\}$ and if $v_i$ concentrates at some $x_0\in B_1,$ then $x_0=0.$
\end{lemma}

\begin{remark}\label{remark:Flucher thesis with si}
(i) Flucher in his paper \cite{Flucher} states and proves this result only for the constant sequence $s_i=1$ for all $i.$ This is not sufficient to prove Proposition \ref{proposition:main domain to ball}, not even for the Moser-Trudinger functional (i.e. $\beta=0$). However, in Flucher's thesis \cite{Flucher Thesis}, this proof is correct. Unfortunately, this thesis is not easily accessible. 
\smallskip

(ii) Note that we make no assumption on the radius of the approximately small disks $\{u_i\geq s_i\},$ nor do we assume any kind of convergence of the $u_i$ towards $G_{\Omega,0}.$
\end{remark}

\begin{proof} Throughout this proof $G=G_{\Omega,0}$ shall denote the Green's function of $\Omega$ with singularity at $0.$
Recall that by assumption there exists real positive numbers $\rho_i$ and $\epsilon_i$ such that for $i\to\infty$
 \begin{equation}
  \label{eq:lemma si instead of 1:lambda to infty and eps over rho to zero}
   \rho_i\to 0\quad\text{ and }\quad \frac{\epsilon_i}{\rho_i}\to 0,
 \end{equation}
satisfying for all $i$ the following inclusion
\begin{equation}
 \label{eq:lemma si instead of 1:ui bigger 1 are asymptotic balls}
  B_{\rho_i-\epsilon_i}\subset\{u_i\geq s_i\}\subset B_{\rho_i+\epsilon_i}.
\end{equation}
\smallskip

\textit{Step 1.} 
Let us define $\lambda_i$, implicitly, by the following equation: 
\begin{equation}
 \label{eq:lemma:definition of rhoi via conf. incenter}
  \rho_i=I_{\Omega}(0)e^{-2\pi \lambda_i},
\end{equation}
that is
$$
  \lambda_i=-\frac{1}{2\pi}\log\left(\frac{\rho_i}{I_{\Omega}(0)}\right).
$$
Note that $\lambda_i\to\infty$ as $i\to\infty.$ We claim that there exists $t_i\geq \lambda_i$ such that 
\begin{equation}
 \label{eq:proof:ti-lambdai goes to zero}
 \lim_{i\to\infty}(t_i-\lambda_i)=0
\end{equation}
and
\begin{equation}
 \label{eq:proof:G bigger ti in ui bigger 1}
  \{G\geq t_i\}\subset\{u_i\geq s_i\}.
\end{equation}
We know from Proposition \ref{proposition:properties of Green's function} (e) that if $t_i\geq 0$ is such that $t_i\to\infty$, then there exists $\sigma_i> 0$ such that
$$
  \lim_{i\to\infty}\frac{\sigma_i}{\tau_i}=0
  \quad\text{ and }\quad 
  B_{\tau_i-\sigma_i}\subset\{G\geq t_i\}\subset B_{\tau_i+\sigma_i},
$$
where $\tau_i=I_{\Omega}(0)e^{-2\pi t_i}.$ In view of \eqref{eq:lemma si instead of 1:ui bigger 1 are asymptotic balls} it is therefore sufficient to choose $t_i$ such that
\begin{equation}
 \label{eq:proof lemma:implicit def. of ti}
  \tau_i+\sigma_i=\rho_i-\epsilon_i\,.
\end{equation}
It remains to show that with this choice \eqref{eq:proof:ti-lambdai goes to zero} is also satisfied. Using \eqref{eq:lemma:definition of rhoi via conf. incenter} and solving the previous equation for $t_i$ explicitly gives that
$$
  t_i=\lambda_i-\frac{1}{2\pi}\log\left(1-\frac{\epsilon_i+\sigma_i}{\rho_i}\right).
$$
Since we know from \eqref{eq:lemma si instead of 1:lambda to infty and eps over rho to zero} that 
$\epsilon_i/\rho_i\to0,$ it is sufficient to show that $\sigma_i/\rho_i\to 0.$ We obtain from \eqref{eq:proof lemma:implicit def. of ti} that
$$
  \frac{\sigma_i}{\tau_i}=\frac{\sigma_i}{\rho_i-\epsilon_i-\sigma_i}
  =\frac{\sigma_i}{\rho_i\left(1-\frac{\epsilon_i}{\rho_i}-\frac{\sigma_i}{\rho_i} \right)}.
$$
Solving this equation for  $(\sigma_i/\rho_i)$ and using that $\epsilon_i/\rho_i\to 0$ and $\sigma_i/\tau_i\to 0$ shows that also $(\sigma_i/\rho_i)\to 0.$ This proves \eqref{eq:proof:ti-lambdai goes to zero}.

\smallskip

\textit{Step 2.} In this step we will show that
\begin{equation}
 \label{eq:proof lemma: step 2 grad norm of ui and ti}
  \int_{\{u_i<s_i\}}|\nabla u_i|^2\geq \frac{s_i^2}{t_i}.
\end{equation}
Let us denote
$$
  U=\{u_i\geq s_i\}\quad\text{ and }\quad V=\{G\geq t_i\}.
$$
From Step 1 we know that $V\subset U$ and by assumption $U\subset \Omega.$ Let $h_i$ be the unique solution of the problem
$$
  \left\{
  \begin{array}{c}
  \Delta h_i=0\quad\text{in }\Omega\backslash V 
  \smallskip \\
  h_i=0\quad\text{on }\delomega\quad\text{and}\quad h_i=1\quad\text{on }  \partial V.
  \end{array}\right.
$$
We see that this is satisfied precisely by $h_i=G/t_i$. Let us define $w_i\in W^{1,2}(\Omega\backslash V)$ by
$$
  w_i=\left\{\begin{array}{lr}
              \frac{u_i}{s_i} &\text{ in }\Omega\backslash U 
              \smallskip \\
              1 &\text{ in } U\backslash V.
             \end{array}\right.
$$
Note that $w_i$ has the same boundary values as $h_i$ on the boundary of $\Omega\backslash V.$ Since $h_i$ minimizes the Dirichlet integral among all such functions we get that
$$
  \int_{\Omega\backslash V}|\nabla h_i|^2\leq \int_{\Omega\backslash V}|\nabla w_i|^2=\int_{\Omega\backslash U}|\nabla w_i|^2=
  \frac{1}{s_i^2}
  \int_{\{u_i<s_i\}}|\nabla u_i|^2.
$$
From Proposition \ref{proposition:properties of Green's function} (a) we know that
$$
  \int_{\Omega\backslash V}|\nabla h_i|^2 =\int_{\{G<t_i\}}\left|\nabla\left(\frac{G}{t_i}\right)\right|^2=\frac{1}{t_i}.
$$
Setting this into the previous inequality proves \eqref{eq:proof lemma: step 2 grad norm of ui and ti}.

\smallskip

\textit{Step 3.} In this step we will define $v_i\in W_{0,rad}^{1,2}(B_1).$ 
Let $\Omega^{\ast}=B_R$ be 
the symmetrized domain and $u_i^{\ast}\in W_{0,rad}^{1,2}(B_R)$ be the radially decreasing symmetric rearrangement of $u_i$. Then there exists $0<a_i<R$ such that
\begin{equation}
 \label{eq:proof lemma step 3:definition of ai}
 \{u_i^{\ast}\geq s_i\}=B_{a_i}.
\end{equation}
Moreover define $0<\delta_i<1$ by 
$\delta_i=e^{-2\pi t_i}.$
At last we can define $v_i$ as
$$
  v_i(x)=\left\{\begin{array}{lr}
                 -\frac{s_i}{2\pi t_i}\log(x) &\text{ if }x\geq \delta_i
                 \smallskip \\
                 u_i^{\ast}\left(\frac{a_i}{\delta_i}x\right) & \text{ if }x \leq \delta_i\,.
                \end{array}\right.
$$
Note that $v_i$ belongs indeed to $W^{1,2}(B_1)$ since the two values coincide if $x=\delta_i$.
\smallskip 

\textit{Step 4.} In this Step we will show that $\|\nabla v_i\|_{L^2(B_1)}\leq \|\nabla u_i\|_{L^2(\Omega)}.$ Let us denote
$$
  A_i=\int_{B_1\backslash B_{\delta_i}}|\nabla v_i|^2\quad\text{ and }
  D_i=\int_{B_{\delta_i}}|\nabla v_i|^2.
$$
A direct calculation gives that
$$
  A_i=s_i^2\int_{\delta_i}^1\frac{2\pi}{(2\pi t_i)^2\,r}dr=\frac{s_i^2}{t_i}.
$$
Using a change of variables and Proposition \ref{prop:Hardy littlewoood modified} (ii) gives that
$$
  D_i=\int_{B_{a_i}}|\nabla u_i^{\ast}|^2=\int_{\{u_i^{\ast}\geq s_i\}} |\nabla u_i^{\ast}|^2\leq \int_{\{u_i\geq s_i\}}|\nabla u_i|^2.
$$
Finally we get that, using \eqref{eq:proof lemma: step 2 grad norm of ui and ti}, that
$$
  \int_{B_1}|\nabla v_i|^2=D_i+A_i\leq \int_{\Omega}|\nabla u_i|^2-\int_{\{u_i<s_i\}}|\nabla u_i|^2+\frac{s_i^2}{t_i}\leq \int_{\Omega}|\nabla u_i|^2.
$$
\smallskip

\textit{Step 5.} In this step we show that
$$
  \lim_{i\to\infty}\frac{a_i}{\delta_i}=I_{\Omega}(0).
$$
Using the fact that $|\{u_i^{\ast}\geq s_i\}|=|\{u_i\geq s_i\}|,$ \eqref{eq:proof lemma step 3:definition of ai} and the hypthesis \eqref{eq:lemma si instead of 1:ui bigger 1 are asymptotic balls} we obtain the inequaly $\rho_i-\epsilon_i\leq a_i\leq \rho_i+\epsilon_i.$
From this we obtain that
$$
  \frac{\rho_i}{\delta_i}\left(1-\frac{\epsilon_i}{\rho_i}\right) 
  \leq\frac{a_i}{\delta_i}\leq
  \frac{\rho_i}{\delta_i}\left(1+\frac{\epsilon_i}{\rho_i}\right).
$$
From the hypothesis \eqref{eq:lemma si instead of 1:lambda to infty and eps over rho to zero} we know that $\epsilon_i/\rho_i\to 0.$ It is therefore sufficient to calculate the limit of $\rho_i/\delta_i$. In view the definition of $\rho_i$ and \eqref{eq:proof:ti-lambdai goes to zero} this is indeed equal to
$$
  \lim_{i\to\infty}\frac{\rho_i}{\delta_i}=\lim_{i\to\infty}I_{\Omega}(0)e^{2\pi(t_i-\lambda_i)}=I_{\Omega}(0),
$$
which proves the statement of this step.

\smallskip

\textit{Step 6 (equality of functional limit).} Let us first show that both $u_i$ and $v_i$ converge to zero almost everywhere. For $u_i$ this holds true by hypothesis. So let $x\in B_1\backslash\{0\}$ be given and note that for all $i$ big enough
$$
  x\geq e^{-2 \pi\sqrt{t_i}}\geq e^{-2\pi t_i}=\delta_i\,.
$$
Therefore we obtain from the definition of $v_i$ and the fact that $s_i$ are bounded, that
$$
  v_i(x)\leq-\frac{s_i}{2\pi t_i}\log\left(e^{-2\pi\sqrt{t_i}}\right)= \frac{s_i}{\sqrt{t_i}}\to 0,
$$
which shows the claim also for $v_i$. In view of Lemma \ref{lemma:a.e. convergence and bounded by 1} it is therefore sufficient to show that
\begin{equation}
 \label{eq:proof lemma:functional equality on ui bigger 1}
  \lim_{i\to\infty}\int_{\{u_i\geq s_i\}}\frac{e^{\alpha u_i^2}-1}{|x|^\beta}
  =I_{\Omega}^{2-\beta}(0)
  \lim_{i\to\infty}\int_{\{v_i\geq s_i\}}\frac{e^{\alpha v_i^2}-1}{|x|^\beta}.
\end{equation}
From Proposition \ref{prop:Hardy littlewoood modified} (i) and the properties of symmetrization we get that for every $i$
$$
  \int_{\{u_i\geq s_i\}}\frac{e^{\alpha u_i^2}-1}{|x|^\beta}\leq
  \int_{\{u_i^{\ast}\geq s_i\}}\frac{e^{\alpha (u_i^{\ast})^2}-1}{(|x|^\beta)^{\ast}}
  =
  \int_{B_{a_i}}\frac{e^{\alpha (u_i^{\ast})^2}-1}{(|x|^\beta)^{\ast}}.
$$
(If $\beta=0$ then the inequality can actally be replaced by an equality, see Kesavan, page 14, equation (1.3.2)). For $i$ big enough $B_{a_i}\subset \Omega,$ and then $(|x|^{\beta})^{\ast}=|x|^{\beta}$ for all $x\in B_{a_i}.$ Making the substitution $x=(a_i/\delta_i)\,y$ gives
$$
  \int_{\{u_i\geq s_i\}}\frac{e^{\alpha u_i^2}-1}{|x|^\beta}\leq
  \left(\frac{a_i}{\delta_i}\right)^{2-\beta}\int_{B_{\delta_i}} \frac{e^{\alpha v_i^2}-1}{|y|^\beta}
  =\left(\frac{a_i}{\delta_i}\right)^{2-\beta}\int_{\{v_i\geq s_i\}} \frac{e^{\alpha v_i^2}-1}{|x|^\beta}.
$$
From Step 5 we therefore get that
$$
  \lim_{i\to\infty}\int_{\{u_i\geq s_i\}}\frac{e^{\alpha u_i^2}-1}{|x|^\beta}\leq
  I_{\Omega}^{2-\beta}(0)\lim\inf_{i\to\infty}\int_{\{v_i\geq s_i\}} \frac{e^{\alpha v_i^2}-1}{|x|^\beta},
$$ 
which proves \eqref{eq:proof lemma:functional equality on ui bigger 1} and hence concludes the proof of the lemma.
\end{proof}
\smallskip

We are now able to prove the main proposition of this section.
\smallskip

\begin{proof}[Proof (Proposition \ref{proposition:main domain to ball}).]
We know from Lemma \ref{lemma:main properties step 1-3 in Flucher} that there exists a sequence, which we call again $\{u_i\}\subset \mathcal{B}_1(\Omega),$ and a sequence $\lambda_i$ such that the properties (i)--(iv) of Lemma \ref{lemma:main properties step 1-3 in Flucher} are satisfied. We now aplly Lemma \ref{lemma:the one we sent to Struwe} to get a new sequence $\{v_l\}\subset  \mathcal{B}_1(\Omega)$ which satisfies properties (a)--(d). Moreover we obtain from property (iv) of $u_i$ and (e) that
$$
  \liminf_{l\to\infty}F_{\Omega}(v_l)\geq F^{\delta}_{\Omega}(0).
$$
Let us again rename $\lambda_{i_l}$ by $\lambda_i$ and $v_l$ by $u_i$. We define
$$
  s_i=\frac{i}{\lambda_i}.
$$
By (a) we obtain that $s_i\leq 1$ for all $i.$ By (b)  the hypothesis $\{u_i\geq s_i\}$ being approximately small disks of Lemma \ref{main lemma si instead of 1. domain to ball} is satisfied. We therefore obtain from Lemma \ref{main lemma si instead of 1. domain to ball} (taking again a subsequence which achieves $\liminf$) that there exists $\{v_i\}\subset W_{0,rad}^{1,2}(B_1)\cap\mathcal{B}_1(B_1)$ such that
$$
  F^{\delta}_{\Omega}(0)\leq\lim_{i\to\infty}F_{\Omega}(u_i)\leq I_{\Omega}^{2-\beta}(0)\liminf_{i\to\infty}F_{B_1}(v_i).
$$
It remains to show that the $\{v_i\}$ has to concentrate at $0.$ If $v_i$ does not concentrate at $0,$ then (cf. Lemma \ref{main lemma si instead of 1. domain to ball}) it does not concentrate at all. 
We therefore get from the concentration compactness alternative Theorem \ref{theorem:concentration alternative for singular moser trudinger} that for some subsequence
$$
  \liminf_{i\to\infty}F_{B_1}(v_i)=\lim_{i\to\infty}F_{B_1}(v_i)=0.
$$
But this leads to the  contradiction  $F_{\Omega}^{\delta}(0)= 0,$ see Remark \ref{remark:F delta zero is not zero}.
\end{proof}

\smallskip

\section{Proof of the Main Theorem}\label{section:proof main theorem}

We now prover Theorem \ref{theorem:intro:Extremal for Singular Moser-Trudinger}.
\smallskip

\begin{proof} In view of Proposition \ref{proposition:sup attained if 0 notin Omegabar} we can assume that $0\in\Omegabar.$ We distinguish to cases.
\smallskip

\textit{Case 1: $0\in\Omega.$}
From Theorems \ref{theorem:concentration formula by domain to ball}, \ref{theorem:supremum of FOmega on Ball} and \ref{theorem:ball to general domain:sup inequality} we know that
$$
  F_{\Omega}^{\delta}(0)=I_{\Omega}^{2-\beta}(0)F_{B_1}^{\delta}(0)< I_{\Omega}^{2-\beta}(0) F_{B_1}^{\sup}\leq F_{\Omega}^{\sup}.
$$
Thus we obtain, using also Proposition \ref{proposition:if u_i concentrates somewhere else than zero}, that $F_{\Omega}^{\delta}(x)<F_{\Omega}^{\sup}$ for all $x\in\Omegabar.$ This implies that maximizing sequences cannot concentrate and the result follows from Theorem \ref{theorem:concentration alternative for singular moser trudinger}.
\smallskip

\textit{Case 2: $0\in\delomega$.} We will show that $F^{\delta}_{\Omega}(0)=0$ in this case. Let $\{u_i\}\subset\mathcal{B}_1(\Omega)$ be a sequence concentrating at $0$ such that
$$
  \lim_{i\to\infty}F_{\Omega}(u_i)=F_{\Omega}^{\delta}(0).
$$
Choose a sequence of bounded smooth open domains $\Omega_n$ which have the property
$$
  0\in\Omega_n,\quad\Omega\subset\Omega_n\quad \forall\,n\in\mathbb{N}\quad\text{ and }\quad\lim_{n\to\infty}|0-\delomega_n|=0,
$$
where $|0-\delomega_n|$ denotes the distance between $0$ and $\delomega_n$. Define for each $n\in\mathbb{N}$ the functions $u_i^n\in\mathcal{B}_1(\Omega_n)$ by extending $u_i$ by zero in $\Omega_n\backslash \Omegabar.$
Note that for each fixed $n$ the sequence $\left\{u_i^n\right\}_{i\in\mathbb{N}}$ concentrates at $0.$ We therefore obtain from Theorem \ref{theorem:concentration formula by domain to ball} that for each $n$
$$
  \lim_{i\to\infty}F_{\Omega}(u_i)=\lim_{i\to\infty}F_{\Omega_n}(u_i^n) \leq F_{\Omega_n}^{\delta}(0)=I_{\Omega_n}^{2-\beta}(0) F_{B_1}^{\delta}(0).
$$
Thus we have shown that for every $n$
$$
  F_{\Omega}^{\delta}(0)\leq I_{\Omega_n}^{2-\beta}(0)F_{B_1}^{\delta}(0).
$$
We now let $n\to\infty$ and use the estimate (see Flucher \cite{Flucher} page 485, proof of Proposition 12 Part 2)
$$
  I_{\Omega_n}(0)\leq 6|0-\delomega_n|,
$$
to obtain that $F_{\Omega}^{\delta}(0)=0.$ So if $0\in\delomega,$ then $F_{\Omega}^{\delta}(x)=0$ for all $x\in\Omegabar.$ We conclude as in Case 1.
\end{proof}

\bigskip

\noindent\textbf{Acknowledgements} We have benefitted from helpful discussions with A. Adimurthi, K. Sandeep and M. Struwe.


\begin{thebibliography}{200}
\small{


\bibitem{Adi-Sandeep} Adimurthi A. and Sandeep K., A singular Moser-Trudinger embedding and its applications, \textit{NoDEA Nonlinear Differential Equations Appl.}, \textbf{13} (2007), no. 5-6, 585--603.
\vspace{-.25cm}

\bibitem{Axler-Bourdon-Ramey}Axler S., Bourdon P. and Ramey W., \textit{ Harmonic function theory}, Second edition, Graduate Texts in Mathematics, 137, Springer-Verlag, New York, 2001.
\vspace{-.25cm}


\bibitem{Carleson-Chang}Carleson L. and Chang S.-Y. A., On the existence of an extremal function for an inequality by J. Moser, \textit{Bull. Sci. Math.}, (2) 110 (1986), no. 2, 113--127.
\vspace{-.25cm}

\bibitem{Csato}Csat\'o G., An isoperimetric problem with density and the Hardy Sobolev inequality in $\re^2$, preprint.  	http://arxiv.org/abs/1410.8041

\vspace{-.25cm}


\bibitem{Flucher}Flucher M., Extremal functions for the Trudinger-Moser inequality in 2 dimensions, \textit{Comment. Math. Helvetici}, \textbf{67} (1992), 471--497.
\vspace{-.25cm}

\bibitem{Flucher Thesis}Flucher M., Extremal functions for the Trudinger-Moser inequality in 2 dimensions, Ph.D. thesis, ETH Z\"urich, 1991.
\vspace{-.25cm}


\bibitem{Kesavan}Kesavan S., \textit{Symmetrization and applications}, Series in Analysis, 3. World Scientific Publishing Co. Pte. Ltd., Hackensack, NJ, 2006.
\vspace{-.25cm}

\bibitem{Lions}Lions P.-L.,
The concentration-compactness principle in the calculus of variations. The limit case. I,
\textit{Rev. Mat. Iberoamericana 1}, (1985), no. 1, 145--201. 
\vspace{-.25cm}

\bibitem{Malchiodi-Martinazzi}Malchiodi A. and Martinazzi L.,
Critical points of the Moser-Trudinger functional on a disk, \textit{
J. Eur. Math. Soc. (JEMS)}, \textbf{16} (2014), no. 5, 893--908.
\vspace{-.25cm}

\bibitem{Moser} Moser J., A sharp form of an inequality by N. Trudinger, \textit{Indiana Univ. Math. J.}, \textbf{20}, (1971), no. 11, 1077--1092.%
\vspace{-.25cm}

\bibitem{Struwe 1}Struwe M., Critical points of embeddings of $H_0^{1,n}$ into Orlicz spaces, \textit{Ann. Inst. H. Poincar\'e Anal. Non Linéaire}, \textbf{5} (1988), no. 5, 425--464. 
\vspace{-.25cm}




\bibitem{Trudinger}Trudinger N.S., On embeddings into Orlicz spaces and some applications, \textit{J. Math. Mech}, \textbf{17} (1967), 473--484.
}
\end{thebibliography}
\end{document}